\documentclass[10pt,letterpaper]{article}

\usepackage[
top=2cm,
bottom=2cm,
left=3cm,
right=3cm
]{geometry}

\usepackage[titletoc,title]{appendix}
\usepackage{amsmath,amssymb,mathrsfs,amssymb,amsthm}
\usepackage{enumerate}
\usepackage{cases}
\usepackage{graphicx}
\usepackage[colorlinks=true, allcolors=blue]{hyperref}
\usepackage{color}
\usepackage{authblk}
 \usepackage{booktabs}
\newcommand{\vep}{\varepsilon}
\newcommand{\R}{\mathbb R}

\definecolor{HW}{rgb}{1,0,0}
\definecolor{HW1}{rgb}{0,0,1}

\numberwithin{equation}{section}
\numberwithin{figure}{section}
\numberwithin{table}{section}

\newtheorem{theorem}{Theorem}[section]
\newtheorem{lemma}{Lemma}[section]
\newtheorem{remark}{Remark}[section]

\title{Eigenvalue Asymptotics in High-Contrast Media
}

\author[1]{Huaian Diao\thanks{E-mail: \texttt{diao@jlu.edu.cn}}}
\author[2]{Long Li\thanks{E-mail: \texttt{long.li@ricam.oeaw.ac.at}}}
\author[2]{Mourad Sini\thanks{E-mail: \texttt{mourad.sini@oeaw.ac.at}}}
\author[1]{Qilong Zhai\thanks{E-mail: \texttt{zhaiql@jlu.edu.cn}}}
\affil[1]{School of Mathematics, Jilin University, Changchun 130012, China}
\affil[2]{RICAM, Austrian Academy of Sciences, Altenberger Stra\ss e 69, 4040 Linz, Austria}

\date{}

\begin{document}

\maketitle

\begin{abstract}
\noindent We investigate the eigenfrequency asymptotics of a bounded acoustic cavity
containing a shrinking high-contrast inclusion. This configuration is motivated by contrast-enhanced ultrasound imaging, in which microbubbles are employed as acoustic contrast agents.

\noindent We identify dimension-dependent material scalings that keep the inclusion-induced resonance in a fixed order-one frequency regime as the
inclusion shrinks. Our main result gives a complete asymptotic description of the spectrum in any prescribed bounded frequency window. Two distinct
spectral mechanisms arise: the background
Dirichlet eigenfrequencies persist as perturbed eigenvalue clusters, while the singular material contrast creates an additional Minnaert eigenvalue branch. The limiting Minnaert frequency is explicit in both dimensions, with a capacitance-based expression in three dimensions and an area-based expression in two dimensions.

\noindent Two-dimensional numerical experiments confirm both the perturbation of
the background Dirichlet eigenfrequencies and the emergence of the
additional Minnaert branch.

\end{abstract}


\section{Introduction and statement of the results}

The Minnaert resonance of a gas bubble in a liquid is one of the most classical examples of a subwavelength acoustic resonance \cite{Minnaert}. It corresponds to the breathing oscillation of the bubble and may occur at a wavelength much larger than the bubble size \cite{CommanderProsperetti1989, Leighton1994}. This strong resonant response is the physical basis for the use of microbubbles in contrast-enhanced ultrasound imaging and is closely related to field amplification and localization near small inclusions \cite{DabrowskiGhandricheSini2021, Senapati-Sini}. They are also extensively used in the modeling and design of effective bulks in metamaterial theory \cite{ACCS-20, AFGLZ-17, AFZ-17}. In the scattering setting, resonances are usually understood as poles of a meromorphically continued outgoing resolvent of the natural Hamiltonian \cite{DM}. For bubbles of arbitrary shape, the natural Hamiltonian has been identified generating a sequence of complex resonances near the real axis, where the Minnaert resonance is the lowest order one, and the leading asymptotic formula for the Minnaert resonance, together with
asymptotic characterizations of the associated scattered field, have been
established in \cite{AZ18, LS-04, LS-042, LS-06, MPS}. All of these works concern scattering resonances or enhanced scattering in an unbounded background.

 In contrast, the present work concerns a real eigenvalue problem arising from a singularly scaled acoustic
inclusion. More precisely, let $D\subset\R^d$, $d=2,3$, be a bounded connected $C^{2}$-smooth domain and let  
\begin{align}\label{eq:def Omega_e}
\Omega_\varepsilon(y_0) := \{\,x : x = y_0 + \varepsilon(y-y_0),\; y\in\Omega\,\},
\qquad 0<\varepsilon\ll1,
\end{align}
with center $y_0 \in D$. Here, $\Omega\Subset D$ is a bounded connected reference inclusion with $C^2-$ smooth boundary $\Gamma:=\partial\Omega$. In order to describe the dimension-dependent balanced regime, we introduce 
\begin{align*}
\tau_\varepsilon^{(d)}:= \begin{cases} \varepsilon^2, & d=3,\\ \varepsilon^2\left|\ln \vep\right|, & d=2. \end{cases}
\end{align*}
The acoustic parameters are then defined by 
\begin{align} \label{eq:17}
\rho_{\vep}(x)=
\begin{cases}
\rho_0, & x\in D\setminus\overline{\Omega_\vep(y_0)},\\
\rho_1\tau_\varepsilon^{(d)}, & x\in \Omega_\vep(y_0),
\end{cases}
\qquad
k_{\vep}(x)=
\begin{cases}
k_0, & x\in D\setminus\overline{\Omega_\vep(y_0)},\\
k_1\tau_\varepsilon^{(d)}, & x\in \Omega_\vep(y_0),
\end{cases}
\end{align}
with positive constants $\rho_0,\rho_1,k_0,k_1$. Hence
\begin{align*}
c_0=\sqrt{\frac{k_0}{\rho_0}},
\qquad c_1=\sqrt{\frac{k_1}{\rho_1}}
\end{align*}
are independent of $\vep$.  We consider signed frequencies $z\in\R$ for the transmission eigenvalue problem
\begin{align}
&\Delta u + z^2c_0^{-2} u = 0 \quad {\rm{in}}\; D \backslash \Omega_\vep(y_0), \label{eq:7} \\
&\Delta u + z^2 c_1^{-2} u = 0 \quad {\rm{in}}\; \Omega_\vep(y_0),\\
& u^+ = u^- \quad {\rm{on}}\; \partial \Omega_\vep(y_0), \\
&\partial_{\nu} u^+ = \frac{\rho_0}{\rho_1\tau_ \vep^{(d)}}\partial_{\nu} u^- \quad {\rm{on}}\; \partial \Omega_\vep(y_0), \label{eq:flux}\\
& u = 0 \quad \mathrm{on}\;\partial D. \label{eq:8}
\end{align}
Here $\nu$ is the outward unit normal to $\Omega_\vep(y_0)$.  We denote by $\Sigma^{(d)}_\vep$ the signed spectrum of \eqref{eq:7}--\eqref{eq:8}. Since the eigenvalue problem depends on the spectral parameter only through $z^2$, $\Sigma^{(d)}_\vep$ is symmetric with respect to the origin. Throughout this paper, we therefore restrict our attention to its positive part.

 The above formulation provides an idealized forward spectral model for bubble-assisted acoustic detection. In contrast-enhanced ultrasound imaging, injected microbubbles act as acoustic contrast agents whose resonant oscillations generate detectable spectral signatures absent from the unperturbed background response. In a bounded acoustic environment, such signatures are naturally reflected in the frequency response of the associated boundary value problem. This motivates the question of whether a shrinking high-contrast inclusion merely perturbs the existing cavity spectrum, or whether it can generate an additional spectral branch of its own.

 We answer this question by studying the spectrum in a fixed bounded frequency window under the simultaneous geometric shrinking and material scaling of the inclusion. Two distinct spectral mechanisms appear. The first is perturbative: the Dirichlet eigenvalues of the unperturbed cavity persist as perturbed eigenvalue clusters. The second is not inherited from the background cavity spectrum: the high-contrast transmission scaling creates an additional Minnaert eigenvalue branch. Thus the problem is not only an eigenvalue perturbation problem for a small inclusion, but also a mechanism for generating a new spectral branch through a shrinking high-contrast resonator.

To state the result, we introduce the following notation. Let
\begin{align*}
\Lambda_D
:=
\left\{
z>0:\; \exists\, v\in H_0^1(D),\ v\ne0,\ 
\Delta v+z^2c_0^{-2}v=0 \text{ in }D
\right\}
\end{align*}
be the set of background Dirichlet frequencies of the cavity. In three dimensions, the limiting Minnaert frequency is
\begin{align}\label{eq:minnaert-3d}
    \omega_M^{(3)}
    :=
    \left(
    \frac{\mathcal C_\Omega k_1}{|\Omega|\rho_0}
    \right)^{1/2},
\end{align}
where
\begin{align}
    \mathcal C_\Omega
    :=
    \int_{\Gamma}
    \left[(S_\Gamma^0)^{-1}1\right]\,d\sigma
\end{align}
is the electrostatic capacitance of the reference inclusion
$\Omega$, with $\Gamma=\partial\Omega$, while in two dimensions, where we have
$\tau_\varepsilon=\varepsilon^2|\log\varepsilon|$, the limiting
Minnaert frequency is
\begin{align}\label{eq:minnaert-2d}
    \omega_M^{(2)}:=
    \left(\frac{2\pi k_1}{|\Omega|\rho_0}
    \right)^{1/2}.
\end{align}
We assume throughout the main result that the limiting Minnaert frequency is
separated from the background Dirichlet spectrum, namely
\begin{align}\label{eq:assup}
    \omega_M^{(d)}\notin \Lambda_D.
\end{align}
Such an assumption is not restricive as we can always tune the bubbles bulk $k_1$ (or its capacitance or volume) to ensure it. 
\begin{theorem}[Finite-window eigenvalue asymptotics]\label{thm:main}
Let $d\in\{2,3\}$, and consider the singularly scaled acoustic
transmission problem \eqref{eq:7}--\eqref{eq:8}. Let $\Sigma_\varepsilon^{(d)}$
denote the set of positive eigenfrequencies of this problem, counted
with algebraic multiplicity. Let $\Lambda_D$ be the set of background
Dirichlet frequencies, and let $\omega_M^{(d)}$ be the limiting
Minnaert frequency defined by \eqref{eq:minnaert-3d} for $d=3$ and by
\eqref{eq:minnaert-2d} for $d=2$. 
Let $I\Subset(0,\infty)$ be a bounded open interval such that
\begin{align*}
    \partial I\cap\bigl(\Lambda_D\cup\{\omega_M^{(d)}\}\bigr)=\emptyset .
\end{align*}
Set
\begin{align*}
    \eta_d(\varepsilon)
    :=
    \begin{cases}
    \varepsilon, & d=3,\\[0.3em]
    |\ln\varepsilon|^{-1}, & d=2.
    \end{cases}
\end{align*}
Under the assumption \eqref{eq:assup}, for all sufficiently small $\varepsilon$, the following statements hold true.

\begin{enumerate}[(a)]
    \item \label{a1} If $\omega_M^{(d)}\in I$, then there exists $\delta>0$ such that $\{s\in I: |s- \omega_M^{(d)}|<\delta \}$, which is disjoint of $\Lambda_D$, contains exactly one eigenfrequency satisfying
\begin{align} \label{eq:asy}
\lambda_{M,\varepsilon}^{(d)}
= \omega_M^{(d)} + O\bigl(\eta_d(\varepsilon)\bigr), \quad \mathrm{as}\; \vep \rightarrow 0.
\end{align}

\item \label{a2} Each background Dirichlet frequency $z_0\in I\cap\Lambda_D$
generates a cluster of eigenfrequencies converging to $z_0$. More
precisely, if
\begin{align} \label{eq:18}
m_D(z_0)
:= \dim\ker\left(-\Delta_D-z_0^2c_0^{-2}\right),
\end{align} 
then there are exactly $m_D(z_0)$ eigenfrequencies of $\Sigma_\varepsilon^{(d)}$, counted with algebraic multiplicity, in a
neighbourhood of $z_0$, and these eigenfrequencies satisfy
\begin{align}\label{eq:31}
\operatorname{dist} 
\left(
\Sigma_\varepsilon^{(d)}\cap B_\delta(z_0),\, z_0
\right)
\le C\eta_d(\varepsilon)
\end{align}
for some constants $C>0$ and $\delta>0$ independent of
$\varepsilon$.
Moreover, no other eigenfrequencies of $\Sigma_\varepsilon^{(d)}$ lie
in $I$ for sufficiently small $\varepsilon$. 
\end{enumerate}
\end{theorem}

\begin{remark}[Choice of the scaling]
The dimension-dependent scaling $\tau_\varepsilon^{(d)}$ is selected to
place the Minnaert eigenvalue in a moderate-frequency regime. This is the
regime in which the shrinking bubble produces a visible resonant response
inside a fixed bounded frequency window. Such an order-one resonance can
serve as a tunable spectral signature of the inclusion and is therefore
relevant for bubble-assisted detection and wave manipulation. In two
dimensions, the additional factor $|\ln\varepsilon|$ is needed to
compensate for the logarithmic low-frequency behaviour of the Green
function.
\end{remark}

Theorem~\ref{thm:main} lies at the intersection of two lines of
spectral perturbation theory. The first concerns classical perturbations by small inhomogeneities, small holes, and singular domain perturbations. In such problems, the small defect acts as a perturbation of a limiting background operator, and one studies how pre-existing spectral data are shifted by the defect; see, for instance, \cite{AK-04,ArrietaHaleHan1991,FLO,RauchTaylor1975}. The Dirichlet-frequency part of Theorem~\ref{thm:main} is of this type: the background Dirichlet eigenvalues persist as perturbed eigenvalue clusters.

The second line concerns singular perturbations that enlarge the limiting spectral picture by creating an additional finite-frequency branch. At the level of spectral structure, the closest comparison is the analysis of small Helmholtz resonators, where an
additional low-frequency eigenvalue can also be generated by a degenerating geometry \cite{CK-23,S-15}. This geometric mechanism has also been used in the analysis of acoustic metamaterials with frequency-dependent effective properties \cite{LS-17}. In that setting, a
resonator volume is connected to the exterior through a thin channel, and
the leading frequency is governed by the volume of the cavity, the length
of the neck, and the cross-sectional area of the channel.
The mechanism
in the present paper is fundamentally different. Here the geometry of the shrinking
inclusion remains simple, and the additional eigenvalue branch is created
by the singular material contrast in the transmission conditions. Consequently, Theorem \ref{thm:main} may be viewed as a material-contrast counterpart of the geometric singular limit for Helmholtz resonators. It simultaneously captures the perturbed cavity eigenvalues and the additional Minnaert branch and proves that these exhaust the spectrum in every prescribed bounded frequency window.

The proof relies on two complementary integral formulations that separately capture the Minnaert branch and the perturbed Dirichlet eigenfrequency clusters. Away from the Dirichlet spectrum, a Lippmann--Schwinger-type reduction based on the Dirichelt Green function identifies the singular contribution generating the Minnaert branch and reveals its dimension-dependent scaling. Near a Dirichlet eigenfrequency, where the Dirichlet Green function becomes singular, we instead use a free-space integral formulation that retains the outer boundary explicitly. Schur-complement reductions and uniform operator estimates, combined with analytic operator-valued function theory, then yield multiplicity-preserving eigenfrequency counts and rule out any additional branches.

The remainder of this paper is organized as follows. In
Section~\ref{section-2}, we introduce the notation and boundary integral
operators used throughout the paper. In Section~\ref{section-3}, we give the proof of the main theorem. In Section~\ref{sec:numerics}, we present two-dimensional numerical experiments supporting the theoretical results.

\section{Notation and boundary integral operators} \label{section-2}

Throughout the paper, let $d\in\{2,3\}$. For $k\in {\mathbb C_+}:=\{z\in \mathbb C: \mathrm{Re}(z)>0\}$, we denote by
$G(\cdot,\cdot;k)$ the outgoing fundamental solution of the Helmholtz operator
$\Delta+k^2$ in $\mathbb R^d$, namely
\begin{align*}
(\Delta_x+k^2)G(x,y;k)=-\delta_y(x).
\end{align*}
More explicitly,
\begin{equation}\label{free-space-green}
G(x,y;k)=
\begin{cases}
\dfrac{e^{ik|x-y|}}{4\pi |x-y|}, & d=3,\\[1.2ex]
\dfrac{i}{4}H_0^{(1)}(k|x-y|), & d=2.
\end{cases}
\end{equation}
In dimension $d=2$, the Hankel function $H_0^{(1)}$ is understood on its principal branch. Consequently, for every $x\neq y$, the mapping
$
k\longmapsto G(x,y;k)$
is analytic in $\mathbb C\setminus(-\infty,0]$. Equivalently, it admits an analytic continuation to the logarithmic Riemann surface of $\mathbb C\setminus\{0\}$. In dimension $d=3$, $G(x,y;k)$ is entire in $k$ for every $x\neq y$. It is known that the following low frequency expansions hold. For $d = 3$,
\begin{align*}
   G(x,y;k) = \frac{1}{4\pi |x-y|} \sum^{\infty}_{j=0} \frac{(ik|x-y|)^j}{j!}, 
\end{align*}
whereas, for $d=2$
\begin{align}\label{eq:49}
-G(x,y;k)
=\frac{1}{2\pi}\,\ln\!\big(k|x-y|\big)+\eta_0
+\sum_{j=1}^{\infty}\Big(b_j \ln\!\big(k|x-y|\big) + s_j\Big)\,(k|x-y|)^{2j},
\end{align}
where 
\begin{align*}
\eta_0=\frac{1}{2\pi}\,(\gamma-\ln 2)-\frac{i}{4},\qquad
b_j=\frac{(-1)^j}{2\pi\,2^{2j}(j!)^{2}},\qquad
s_j=b_j\!\left(\gamma-\ln 2-\frac{i\pi}{2}-\sum_{n=1}^{j}\frac{1}{n}\right),
\end{align*}
and $\gamma$ is the Euler constant. For a bounded measurable set $U\subset\mathbb R^d$ and a Lipschitz
surface, or curve when $d=2$, $\Sigma\subset\mathbb R^d$, we define the
free-space volume potential, single-layer potential, and adjoint Neumann--Poinca\'e operator by
\begin{align}
(\mathcal N_U^k g)(x)
&:=\int_U G(x,y;k)g(y)\,dy,
&& x \in \mathbb R^d, \label{free-volume-potential}\\
(\mathcal {SL}_\Sigma^k\varphi)(x)
&:=\int_\Sigma G(x,y;k)\varphi(y)\,d\sigma(y),
&& x\in \mathbb R^d\setminus \Sigma, \label{free-single-layer}\\
(\mathcal K_\Sigma^{k,*}\varphi)(x)
&:=\operatorname{p.v.}\int_\Sigma
\partial_{\nu_x}G(x,y;k)\varphi(y)\,d\sigma(y),
&& x\in \Sigma. \label{free-np}
\end{align}
Here $\nu$ denotes the chosen unit normal on $\Sigma$, and the normal
derivative in \eqref{free-np} is taken with respect to the $x$-variable. 

We shall also use the corresponding operators associated with the
Dirichlet Green function in the bounded cavity $D$. Let $k^2$ be outside
the Dirichlet spectrum of $-\Delta$ in $D$. The Dirichlet Green function
$G_D(\cdot,\cdot;k)$ is defined by
\begin{align}
(\Delta_x+k^2)G_D(x,y;k)&=-\delta_y(x),
&& x\in D, \label{dirichlet-green-equation}\\
G_D(x,y;k)&=0,
&& x\in \partial D. \label{dirichlet-green-bc}
\end{align}
Equivalently, one may write
\begin{equation}\label{eq:1}
G_D(x,y;k)=G(x,y;k) - H_D(x,y;k),
\end{equation}
where $H_D(\cdot,y;k)$ solves the homogeneous Helmholtz equation in $D$
with boundary data $G(\cdot,y;k)$ on $\partial D$.
For $U\Subset D$ and $\Sigma\Subset D$, we define the Dirichlet volume
potential, Dirichlet single-layer potential, and Dirichlet
Neumann--Poincar\'e operator by
\begin{align}
(\mathcal N_{D,U}^k g)(x)
&:=\int_U G_D(x,y;k)g(y)\,dy,
&& x\in D, \label{dirichlet-volume-potential}\\
(\mathcal {SL}_{D,\Sigma}^k\varphi)(x)
&:= \int_\Sigma G_D(x,y;k)\varphi(y)\,d\sigma(y),
&& x\in D\setminus \Sigma, \label{dirichlet-single-layer}\\
(\mathcal K_{D,\Sigma}^{k,*}\varphi)(x)
&:=\operatorname{p.v.}\int_\Sigma
\partial_{\nu_x}G_D(x,y;k)\varphi(y)\,d\sigma(y),
&& x\in \Sigma. \label{dirichlet-np}
\end{align}
By construction, the potentials $\mathcal N_{D,U}^k g$ and
$\mathcal {SL}_{D,\Sigma}^k\varphi$ satisfy the homogeneous Dirichlet
condition on $\partial D$.

Given $\vep > 0$ and $y_0 \in \R^3$, define
\begin{align} \label{eq:scal}
\Phi_\varepsilon(x;y_0):=y_0+\varepsilon(x-y_0).
\end{align}
For two Banach spaces $X$ and $Y$, denote the space of all linear bounded mapping from $X$ to $Y$ by $\mathcal L(X,Y)$. For simplicity, $\mathcal L(X,X)$ is also denoted by $\mathcal L(X)$. From now on, $\mathbb I$ denotes an identity operator in various spaces, and the constants may be different at different places.

\section{Proof of Theorem \ref{thm:main}} \label{section-3}

This section is devoted to the proof of Theorem \ref{thm:main}. The proofs of statements \eqref{a1} and \eqref{a2} are given in section \ref{sec:p1} and \ref{sec:p2}.

\subsection{Proof of statement (\ref{a1}) in Theorem \ref{thm:main}}\label{sec:p1}

We begin with the following observation: when $z \notin \Lambda_D$, any solution $u \in H^1_0(D)$ in \eqref{eq:7}--\eqref{eq:8} solves 
\begin{align*}
u(x) & = \alpha z^2 \int_{\Omega_\vep(y_0)} G(x,y;z/c_0) u(y)ds(y)\\
&- \beta^{(d)}_\vep \int_{\partial \Omega_\vep(y_0)}\frac{\rho_0}{\rho_1\tau^{(d)}_\vep}\partial_\nu u^-(y) G(x,y;z/c_0)ds(y), \quad x\in D \backslash \partial \Omega_\vep(y_0).
\end{align*}
Here, 
\begin{align} \label{eq:16}
\alpha:= \frac{1}{c^{2}_1} - \frac{1}{c^{2}_0},\qquad \beta^{(d)}_\vep:= 1 - \frac{\rho_1\tau^{(d)}_\vep}{\rho_0}.
\end{align}
Clearly, we find
\begin{align} \label{eq:9}
\begin{bmatrix}
&\mathbb I - z^2 \alpha \mathcal N^{z/c_0}_{D,\Omega_\vep(y_0)} & \beta^{(d)}_\vep  \mathcal  {SL}^{z/c_0}_{D,\partial \Omega_\vep(y_0)} \\
& - z^2 \alpha \partial_\nu  \mathcal N^{z/c_0}_{D,\Omega_\vep(y_0)} & \frac{\mathbb I} 2 +  \mathcal  K^{z/c_0,*}_{D,\partial \Omega_\vep(y_0)}  + \frac{\rho_1\tau^{(d)}_\vep}{\rho_0}\left(\frac{\mathbb I} 2 -  \mathcal K^{z/c_0,*}_{D,\partial \Omega_\vep(y_0)}\right)
\end{bmatrix} 
\begin{bmatrix}
u|_{\Omega_\vep(y_0)}\\
\frac{\rho_0}{\rho_1\tau^{(d)}_\vep}\partial_{\nu} u^-|_{\partial \Omega_\vep(y_0)}\\
\end{bmatrix} = \begin{bmatrix}
0\\
0
\end{bmatrix}. 
\end{align}
Since $\omega^{(d)}_M\notin\Lambda_D$, we may choose $\delta>0$ sufficiently small such that 
\begin{align} \label{eq:15}
\overline{B_\delta(\omega^{(d)}_M)}\cap\Lambda_D=\emptyset.
\end{align}
In what follows, we restrict the spectral parameter $z$ to $B_\delta(\omega_M)$.
By the Born series inversion method, it is easy to verify that
\begin{align*}
\mathbb I - z^2 \alpha \mathcal N^{z/c_0}_{D,\Omega_\vep(y_0)}
\end{align*}
is invertible in $\mathcal L(L^2(\Omega_\vep(y_0)))$ for all sufficiently small $\vep$. Therefore, it remains to investigate the injectivity of  
\begin{align*}
 \mathcal B^{\mathrm{sc}}_{\vep}(z):=\mathcal B_{\vep}(z) + z^2 \alpha \partial_\nu  \mathcal N^{z/c_0}_{D,\Omega_\vep(y_0)}\left(\mathbb I - z^2 \alpha \mathcal N^{z/c_0}_{D,\Omega_\vep(y_0)}\right)^{-1} \beta^{(d)}_\vep  \mathcal  {SL}^{z/c_0}_{D,\partial \Omega_\vep(y_0)}   
\end{align*}
in $L^2(\Omega_\vep(y_0)) \times L^2(\partial \Omega_\vep(y_0))$, where 
\begin{align*}
\mathcal B_{\vep}(z):=\frac{\mathbb I} 2 + \mathcal K^{z/c_0,*}_{D,\partial \Omega_\vep(y_0)}  + \frac{\rho_1 \tau^{(d)}_\vep}{\rho_0}\left(\frac{\mathbb I} 2 - \mathcal K^{z/c_0,*}_{D,\partial \Omega_\vep(y_0)}\right).
\end{align*}
Furthermore, building upon the decomposition \eqref{eq:1}, and pulling the operator back from
$\overline{\Omega_\varepsilon}$ to $\overline{\Omega}$ via the scaling transformation \eqref{eq:scal}, it can be deduced that the injectivity of $\mathcal B^{\mathrm{sc}}_\vep$ is equivalent to the injectivity of the operator below
\begin{align} \label{eq:2}
 &\frac{\mathbb I} 2 + \mathcal K_{D,\Gamma}^{\vep z/c_0,*}  + \mathcal K_{\vep,\mathrm{bod}}^{z} + \frac{\rho_1 \tau^{(d)}_\vep}{\rho_0}\left(\frac{\mathbb I} 2 -\mathcal K^{\vep z/c_0,*}_{D,\Gamma} - \mathcal K^{z}_{\vep,\mathrm{bod}}\right) + \mathcal M_\vep^{(0)}(z). 
\end{align}
where 
\begin{align}
&M_\vep^{(0)}(z):=\notag\\
&z^2 \alpha \left(\vep\partial_\nu \mathcal N_{D,\Omega}^{\vep z/c_0} + \mathcal N^{z,\mathrm{tra}}_{\vep,\mathrm{bod}}\right)\left(\mathbb I - z^2 \alpha(\vep^2\mathcal N_{D,\Omega}^{\vep z/c_0} + \mathcal N_{\vep,\mathrm{bod}}^{z})\right)^{-1}\beta^{(d)}_\vep \left[\vep\mathcal SL_{D,\Gamma}^{\vep z/c_0} + \mathcal SL_{\vep,\mathrm{bod}}^{z}\right]. \label{eq:50}
\end{align}
Here,
\begin{align*}
&\left(\mathcal K^{z}_{\vep,\mathrm{bod}} f \right)(x):= { -}\vep^{d-2} \int_{\Gamma}\nu(x)\cdot \nabla_x H_D(\Phi_\vep(x;y_0), \Phi_\vep(y;y_0)); z/c_0) f(y) ds(y), \quad x \in \Gamma.\\
&\left(\mathcal SL^{z}_{\vep,\mathrm{bod}} f \right)(x):= -\vep^{d-1} \int_{\Gamma}H_D(\Phi_\vep(x;y_0), \Phi_\vep(y;y_0)); z/c_0) f(y) ds(y), \quad\ x \in \Omega,\\
&\left(\mathcal N^{z}_{\vep,\mathrm{bod}} f \right)(x):= {-}\vep^{d} \int_{\Omega} H_D(\Phi_\vep(x;y_0), \Phi_\vep(y;y_0)); z/c_0) f(y) ds(y), \quad x \in \Omega,\\
&\left(\mathcal N^{z,\mathrm{tra}}_{\vep,\mathrm{bod}} f \right)(x):= {-}\vep^{d-1} \int_{\Omega} \nu(x)\cdot \nabla_x H_D(\Phi_\vep(x;y_0), \Phi_\vep(y;y_0)); z/c_0) f(y) ds(y), \quad x \in \Gamma.
\end{align*}
Now we collect two auxiliary results concerning those integral operators appear in formula \eqref{eq:2}. They are stated in the Lemma \ref{lem:equilibrium-projection} and Lemma \ref{le:1}, respectively.

\begin{lemma}\label{lem:equilibrium-projection}
Let
$\varphi_0$ be chosen as follows:
\begin{align*}
    \varphi_0 = \left(\mathcal S_\Gamma^0\right)^{-1}1 \quad \text{if } d=3,
\end{align*}
while, if $d=2$, we can define $\varphi_0$ analogously (see Remark \ref{re1} below for its construction ), and $\varphi_0$ forms a basis for
\begin{align*}
\ker\left(\frac12 \mathbb I + \mathcal K_{\Gamma}^{0,*}\right).
\end{align*}
Then, in both cases,
\begin{align*}
\left(\frac12 \mathbb I + \mathcal K_{\Gamma}^{0,*}\right)\varphi_0=0.
\end{align*}
Define
\begin{align}\label{eq:13}
\mathcal P\psi
:= \frac{(\psi,1)_{L^2(\Gamma)}}{(\varphi_0,1)_{L^2(\Gamma)}}\,\varphi_0, \qquad \mathcal Q:= I - \mathcal P.
\end{align}
Then $\mathcal P$ is a rank-one projection and
\begin{align*}
\mathcal P \mathcal K_{\Gamma}^{0,*} = \mathcal K_{\Gamma}^{0,*}\mathcal P = -\frac12\mathcal P.
\end{align*}
Consequently,
\begin{align} \label{eq:14}
 \mathcal T\left(\frac12 I + \mathcal K_0^*\right)
 =
 \left(\frac12 I + \mathcal K_0^*\right)\mathcal T,
 \qquad
 \mathcal T\in\{\mathcal P,\mathcal Q\}.
\end{align}
\end{lemma}

\begin{remark} \label{re1}
In two dimensions, the static single-layer operator $S^0_\Gamma$ may fail to be
invertible. Due to the low-frequency expansion
\begin{align*}
G(x,y;k)
= - \frac{1}{2\pi}\ln(|x-y|) - \eta_k +
O\bigl(k^2|x-y|^2\ln(k|x-y|)\bigr),
\end{align*}
where
\begin{align}
\eta_k =
\frac{1}{2\pi}(\ln k+\gamma-\ln2)-\frac{i}{4},
\end{align}
the modified single-layer operator
\begin{align*}
\left(\widehat {\mathcal S}^0_{k,\Gamma} \psi\right)(x)
:= \left(\mathcal S^0_\Gamma \psi\right)(x) + \eta_k(\psi,1)_{L^2(\Gamma)}, \quad x\in \Gamma
\end{align*}
is invertible in $\mathcal L(L^2(\Gamma), H^1(\Gamma))$ (see Lemma B.1 in \cite{AZ18}). Let $\varphi_0$ be the equilibrium density introduced above and write
\begin{align*}
    \mathcal S_\Gamma^0 \varphi_0 = \gamma_0
    \quad\mathrm{on}\;\Gamma, \quad C_\Gamma:= (\varphi_0,1)_{L^2(\Gamma)}.
\end{align*}
Then, we set
\begin{align*}
    \widehat {\mathcal S}_{\Gamma,k}^0 \varphi_0
    =
    \gamma_0 + \eta_k C_\Gamma.
\end{align*}
Since $\gamma_0$ and $C_\Gamma$ are real, $C_\Gamma\neq0$, and
$\operatorname{Im}\eta_k=-1/4$, we have
\begin{align*}
\gamma_0 + \eta_k C_\Gamma\neq 0.
\end{align*}
Consequently,
\begin{align*}
    \left(\widehat {\mathcal S}^0_{\Gamma,k}\right)^{-1}1
    =
    \frac{\varphi_0}{\gamma_0+\eta_k C_\Gamma}.
\end{align*}
Thus $\left(\widehat {\mathcal S}^0_{\Gamma,k}\right)^{-1}1$ provides the two-dimensional analogue of the
three-dimensional quantity $\left(\mathcal S^0_\Gamma\right)^{-1}1$ up to a scalar normalization
of the same equilibrium density $\varphi_0$. For later use, in two dimensions, we set
\begin{align*}
    \varphi_0:= \left(\widehat {\mathcal S}^0_{\Gamma,k}\right)^{-1} 1.
\end{align*}
\end{remark}

\begin{lemma} \label{le:1}
Assume that $y_0 \in D $. Let $V$ be a fixed bounded open subset of the domain of analyticity
of $H_D(y_0,y_0,z)$ with respect to $z$, and let $\varepsilon>0$ be sufficiently small.
\begin{enumerate}[(a)]
\item \label{d1}
The operator-valued maps
\begin{align*}
z\longmapsto\mathcal K_{\varepsilon,\mathrm{bod}}^z\qquad
z\longmapsto\mathcal{SL}_{\varepsilon,\mathrm{bod}}^z,\qquad
z\longmapsto\mathcal \mathcal N_{\varepsilon,\mathrm{bod}}^z, \qquad z \longmapsto \mathcal N_{\vep,\mathrm{bod}}^{z,\mathrm{tra}}
\end{align*}
are analytic from $V$ into $\mathcal L(L^2(\Gamma))$, $\mathcal L(L^2(\Gamma),L^2(\Omega))$, $\mathcal L(L^2(\Omega))$ and $\mathcal L(L^2(\Omega),L^2(\Gamma))$, respectively. Moreover, for
every compact set $K\Subset V$, there exists a constant $C_K>0$,
independent of $\varepsilon$, such that
\begin{align}\label{eq:41}
&\sup_{z\in K}
\left\|\mathcal K_{\varepsilon,\mathrm{bod}}^z
\right\|_{\mathcal L(L^2(\Gamma))}
\leq C_K\varepsilon^{d-1},\\
&\sup_{z\in K}
\left\| \mathcal SL_{\varepsilon,\mathrm{bod}}^z
\right\|_{\mathcal L(L^2(\Gamma),L^2(\Omega))}
\leq C_K\varepsilon^{d-1},\label{eq:54}\\
&\sup_{z\in K}
\left\|\mathcal N_{\varepsilon,\mathrm{bod}}^z
\right\|_{\mathcal L(L^2(\Omega))}
\leq C_K\varepsilon^{d},\\
&\sup_{z\in K}
\left\|\mathcal N_{\varepsilon,\mathrm{bod}}^{z,\mathrm{tra}}
\right\|_{\mathcal L(L^2(\Omega),L^2(\Gamma))}
\leq C_K\varepsilon^{d}. \label{eq:55}
\end{align}

\item \label{d2} For $g \in L^2(\Gamma)$, we have that
\begin{align}\label{eq:5}
&\left(\mathcal K_{\vep, \mathrm{bod}}^{z} g, 1\right)_{L^2(\Gamma)} =  {\vep^d}\frac{z^2}{c^2_0}H_{D}(y_0, y_0; z/c_0)|\Omega|\left(g,1\right)_{L^2(\Gamma)} + O(\vep^{d+1}),\\
&\left(\mathcal SL_{\vep, \mathrm{bod}}^{z} g, 1\right)_{L^2(\Omega)} = -  {\vep^{d-1}}H_{D}(y_0, y_0; z/c_0)|\Omega|\left(g,1\right)_{L^2(\Gamma)} + O(\vep^{d}), \label{eq:57}
\end{align}
{as} $\vep \rightarrow 0$.
\end{enumerate}
\end{lemma}

\begin{proof}
We first prove statement~\eqref{d1}. 
For sufficiently small $\varepsilon$, the function
\begin{align*}
z\longmapsto
H_D\bigl(
\Phi_\varepsilon(x;y_0),
\Phi_\varepsilon(y;y_0);
z/c_0
\bigr)
\end{align*}
is analytic on $U$ for every $x,y\in\Gamma$. Moreover, $H_D$ and
its derivatives are uniformly bounded in the spatial variables for
$z$ in compact subsets of $V$.
By the chain rule,
\begin{align*}
\nabla_x
H_D\bigl(
\Phi_\varepsilon(x;y_0),
\Phi_\varepsilon(y;y_0);
z/c_0
\bigr)
=
\varepsilon
\nabla_1 H_D\bigl(
\Phi_\varepsilon(x;y_0),
\Phi_\varepsilon(y;y_0);
z/c_0
\bigr),
\end{align*}
where $\nabla_1$ denotes differentiation with respect to the first
spatial variable. Hence, for every compact set $K\Subset V$,
\begin{align*}
\sup_{\substack{x,y\in\Gamma\\ z\in K}}
\left|
\nabla_1 H_D\bigl(
\Phi_\varepsilon(x;y_0),
\Phi_\varepsilon(y;y_0);
z/c_0
\bigr)
\right|
\leq C_K.
\end{align*}
It follows from the definition of
$\mathcal K_{\varepsilon,\mathrm{bod}}^z$ and the
Cauchy--Schwarz inequality that \eqref{eq:41} holds.
Since the kernel is analytic in $z$ and locally uniformly bounded,
differentiation under the integral sign shows that
\begin{align*}
z\longmapsto
\mathcal K_{\varepsilon,\mathrm{bod}}^z
\end{align*}
is analytic as an
$\mathcal L(L^2(\Gamma))$-valued function. 

The remaining analyticity assertions of statement \eqref{d1} and the estimates \eqref{eq:54}--\eqref{eq:55} follow by the same argument.

We now prove statement \eqref{d2}. It is easy to verify that as $\vep \rightarrow 0$,
\begin{align}
&H_D(\Phi_\vep(x;y_0), \Phi_\vep(y;y_0); z/c_0) =   H_D(y_0,y_0; z/c_0) + O(\vep).\label{eq:3}
\end{align}
Furthermore, since
\begin{align*}
  \left(\Delta_x + \vep^2z^2/c_0^2\right)  H_D(\Phi_\vep(x;y_0),\Phi_\vep(y;y_0);z/c_0) = 0\quad \mathrm{in}\; \Omega\;\; \textrm{for each fixed }\; y,
\end{align*}
we have 
\begin{align}
& \left(\mathcal K_{\vep, \mathrm{bod}}^{z} g, 1\right)_{L^2(\Gamma)} \\
&= {-} \vep^{d-2}\int_{\Gamma}\int_\Gamma \nu(x) \cdot \nabla_x H_D(\Phi_\vep(x;y_0),\Phi_\vep(y;y_0);z/c_0) g(y) ds(y) ds(x) \notag\\
& = \vep^d \int_{\Omega}\frac{z^2}{c^2_0}H_{D}(\Phi_\vep(x;y_0),\Phi_\vep(y;y_0);z/c_0) dx \int_{\Gamma} g(y)ds(y). \notag
\end{align}
This, together with \eqref{eq:3} yields the asymptotic expansion \eqref{eq:5}.

Furthermore, using \eqref{eq:3} again, it can be seen that \eqref{eq:57} holds.
\end{proof}

Now we are ready to show statement (\ref{a1}) in Theorem \ref{thm:main}.

\begin{proof}[Proof of statement (\ref{a1}) in Theorem \ref{thm:main}]

By a straightforward calculation, we have: for any $\phi \in L^2(\Gamma)$,
\begin{align*}
\left\langle\partial_\nu \mathcal N_{D,\Omega}^{\vep z/c_0}\mathcal SL_{D,\Gamma}^{\vep z/c_0} \phi, 1\right\rangle_{L^2(\Gamma)} & =  -\frac{\vep^2z^2}{c^2_0}\int_\Omega \left(N_{D,\Omega}^{\vep z/c_0} SL_{D,\Gamma}^{\vep z/c_0} \phi\right) (y)dy \\
& + \frac{c^2_0}{\vep^2 z^2}\left\langle \frac{1}{2}\phi+ K^{\vep z/c_0,*}_{D,\Gamma} \phi, 1\right\rangle_{L^2(\Gamma)}
\end{align*}
and
\begin{align*}
\left\langle\partial_\nu \mathcal N_{D,\Omega}^{\vep z/c_0}\mathcal SL_{\vep,\mathrm{bod}}^{z}\phi, 1\right\rangle_{L^2(\Gamma)} = \int_\Omega  -\frac{\vep^2z^2}{c^2_0}\left(N_{D,\Omega}^{\vep z/c_0} \mathcal SL_{\vep,\mathrm{bod}}^{z}\phi\right)(y) - \left(\mathcal SL_{\vep,\mathrm{bod}}^{z}\phi\right)(y)dy.
\end{align*}
This, together with \eqref{eq:49}, \eqref{eq:14} and statement \eqref{d1} of Lemma \ref{le:1} yields the following the decomposition of the operator $M_\vep^{(0)}(z)$, given by \eqref{eq:50}:
\begin{align*}
 M_\vep^{(0)}(z) = c_0^2\alpha \mathcal P \left(\frac{1}{2} + K^{\vep z/c_0,*}_{D,\Gamma}\right) + \mathcal M_{\vep}^{(1)}(z) + \mathcal M_{\vep}^{(2)}(z) +  \mathcal M_{\vep,\mathrm{res}}^{(3)}(z),
\end{align*}
where $\mathcal M_{\vep}^{(1)}$ and $\mathcal M_{\vep}^{(2)}$ are defined by 
\begin{align} \label{eq:47}
  \mathcal M_{\vep}^{(1)}(z):= -\vep \alpha z^2 \frac{\langle SL_{\vep,\mathrm{bod}}^{z}\phi,1\rangle_{L^2(\Omega)}}{(\varphi_0,1)_{L^2(\Gamma)}}\,\varphi_0,\qquad \mathcal M_{\vep}^{(2)}(z):=\vep^2 \alpha \mathcal Q \partial_\nu \mathcal N_{D,\Omega}^{\vep z/c_0} \mathcal SL_{D,\Gamma}^{\vep z/c_0},
\end{align}
and $\mathcal M_{\vep, \mathrm{res}}^{(2)}$ satisfies
\begin{align}\label{eq:48}
\left\|\mathcal M_{\vep,\mathrm{res}}^{(2)}(z)\right\|_{\mathcal L(L^2(\Gamma))} \le C \left(\tau^{(d)}_\vep\right)^2.
\end{align}

On the other hand, by the representation of \eqref{eq:2}, it follows that a number $\lambda$ is the eigenvalue of \eqref{eq:7}--\eqref{eq:8}, if and only if there exists $g_\vep \in L^2(\Gamma)$ such that 
\begin{equation}\label{original-characteristic-equation}
\left[\frac{\rho_1 \vep^2}{\rho_0}\mathbb I  + \beta^{(3)}_\vep \left(\frac{\mathbb I}2 + \mathcal K_{D,\Gamma}^{\vep z/c_0,*}  + \mathcal K_{\vep,\mathrm{bod}}^{z}\right) + M_\vep^{(0)}(z)\right]g_\vep = 0 \quad \textrm{on}\; \Gamma.
\end{equation}
With the decomposition $\psi_\vep = \mathcal P \psi_\vep + \mathcal Q \psi_\vep$, \eqref{original-characteristic-equation} is equivalent to 
\begin{align} \label{eq:137}
\mathcal F(z(\vep)) g_\vep = \mathcal P g_\vep, 
\end{align}
where $\mathcal F(z(\vep))$ is defined by
\begin{align} \label{eq:142}
\mathcal F(z(\vep)) := \frac{\rho_1 \vep^2}{\rho_0}\mathbb I + \mathcal P + \beta^{(3)}_\vep\left(\frac{\mathbb I}2 + \mathcal K_{D,\Gamma}^{\vep z/c_0,*} + \mathcal K_{\vep,\mathrm{bod}}^{z}\right) + M_\vep^{(0)}(z).
\end{align}
Here, the projection operator is given by \eqref{eq:13}.
With the aid of the fact that $\mathcal P + {\mathbb I}/2 + \mathcal K_{D,\Gamma}^{0,*}$  has an inverse in $\mathcal L(L^{2}(\Gamma))$, we have 
\begin{align} 
\left\|\left(\mathcal F(z(\vep))\right)^{-1} \right\|_{\mathcal L \left(L^{2}(\Gamma)\right)} \le C.\notag
\end{align}
Thus, \eqref{eq:137} can be rewritten as 
\begin{align*}
\mathcal Q g_\vep = \left(\mathcal F(z(\vep))\right)^{-1} \mathcal P g_\vep - \mathcal P g_\vep.
\end{align*}
This, together with the definition of $\mathcal P$ and $\mathcal Q$ gives 
\begin{align} \label{eq:141}
\langle\left(\mathcal F(z(\vep))\right)^{-1} \mathcal P g_\vep, 1\rangle_{L^2(\Gamma)} =\langle \mathcal P g_\vep, 1\rangle_{L^2(\Gamma)}.
\end{align}
Setting 
\begin{align}\label{def-l}
l_\vep:= \left(\mathcal F(z(\vep))\right)^{-1} \varphi_0,
\end{align}
then we rewrite \eqref{eq:141} as
\begin{align}\label{l-psi-0}
\langle l_\vep, 1\rangle_{L^2(\Gamma)}
= 1.
\end{align}
In conjunction with \eqref{eq:142}, \eqref{def-l} and \eqref{l-psi-0}, we arrive at 
\begin{align}
\left[\frac{\rho_1 \vep^2}{\rho_0}\mathbb I + \beta^{(3)}_\vep\left(\frac{\mathbb I}2 + \mathcal K_{D,\Gamma}^{\vep z/c_0,*} + \mathcal K_{\vep,\mathrm{bod}}^{z}\right) + M_\vep^{(0)}(z)\right] l_\vep = 0. \label{eq:51}
\end{align}

The rest of the proof is divided into two cases: the first case involves three dimensional case, and the second case addresses two dimensional case.

\medskip
\noindent\textbf{Case 1: Three-dimensional case.}
Combining \eqref{eq:47}, \eqref{eq:48}, \eqref{eq:51} and the asymptotics of $\mathcal K_{D,\Gamma}^{z,*}$ around $0$ gives
\begin{align}
0 = \left[\frac{\rho_1 \vep^2}{\rho_0} \mathbb I + \left(\beta^{(3)}_\vep \mathbb I + \alpha c_0^2\mathcal P\right)\left(\frac 12 + \mathcal K_{D,\Gamma}^{0,*} + \mathcal K^{z}_{\vep,\mathrm{bod}} + \frac{(\vep z)^2}{c^2_0}\mathcal K_{D,\Gamma,2}^{0} \right) + \mathcal M^{(2)}_\vep(z) + \mathcal R^{(1)}_{\textrm{res},\vep}(z)\right]l_\vep,\label{eq:143} 
\end{align}
where 
\begin{align}\label{eq:144}
\left\|\mathcal R^{(1)}_{\textrm{res},\vep}(z)\right\|_{\mathcal L\left(L^2(\Gamma)\right)} \le C \vep^3.
\end{align}
Here, 
\begin{align*}
\mathcal K_{D,\Gamma,2}^{0}:= \frac{-1}{8\pi } \int_{\Gamma}|x-y|^{-1}(x-y)\cdot \nu(x) \phi(y) d\sigma(y).
\end{align*}
With the aid of \eqref{eq:14} and \eqref{eq:143}, we have 
\begin{align}
&\left(\frac{1}2 +\mathcal K_{D,\Gamma}^{0,*}\right) \mathcal Q l_\vep + \mathcal Q \widetilde R^{(1)}_{\textrm{res},\vep}(z) \mathcal Ql_\vep + \mathcal M_\vep^{(2)}(z)\mathcal Q l_\vep = -\mathcal Q \widetilde R^{(1)}_{\textrm{res},\vep}(z) \mathcal P l_\vep - \mathcal M_\vep^{(2)}(z)\mathcal P l_\vep , \label{eq:44}\\
&\mathcal P \widetilde R^{(1)}_{\textrm{res},\vep}(z) \mathcal Pl_\vep + \mathcal P \widetilde R^{(1)}_{\textrm{res},\vep}(z) \mathcal Ql_\vep = 0. \label{eq:45}
\end{align}
Here,
\begin{align*}
 \widetilde R^{(1)}_{\textrm{res},\vep}:=   \frac{\rho_1 \vep^2}{\rho_0}\mathbb I-\frac{\rho_1 \vep^2}{\rho_0}\left(\frac{1}2 + \mathcal K_{D,\Gamma}^{0,*}\right) + \left(1-\frac{\rho_1 \vep^2 }{\rho_0}\right)\left(\frac{z^2\vep^2}{ c^2_0} \mathcal K^{0}_{D,\Gamma,2} + \mathcal K^{z}_{\vep,\mathrm{bod}}+ \mathcal R^{(1)}_{\textrm{res},\vep}(z)\right).
\end{align*}
It follows from \eqref{eq:41}, \eqref{eq:5} and  \eqref{eq:144} that \begin{align}\label{eq:146}
\left\|\mathcal Q\widetilde R^{(1)}_{\textrm{res},\vep}(z)\right\|_{\mathcal L(L^2(\Gamma))} \le C \vep^2, \qquad\left\|\mathcal P\widetilde R^{(1)}_{\textrm{res},\vep}(z)\right\|_{\mathcal L(L^2(\Gamma))} \le C \vep^2.
\end{align}
Furthermore, with the aid of the fact that $\left(1/2 + \mathcal K_{D,\Gamma}^{0,*}\right)\mathcal Q$ is invertible in $\mathcal L(\mathcal Q(L^2(\Gamma)))$, we have 
\begin{align*}
\left\|\left[\left(\frac{1}2 +\mathcal K_{D,\Gamma}^{0,*}\right) \mathcal Q + \mathcal Q \widetilde R^{(1)}_{\textrm{res},\vep}(z) \mathcal Q + \mathcal M_\vep^{(2)}(z)\mathcal Q \right]^{-1}\right\|_{\mathcal L(L^2(\Gamma))} \le C.  
\end{align*}
Combining this with \eqref{eq:5} and \eqref{eq:146}, and applying
Lemma~\ref{lem:equilibrium-projection}, we conclude that
\eqref{eq:44}--\eqref{eq:45} admit a nontrivial solution if and only if
\begin{align} \label{eq:46}
\left(\frac{\vep^2 \rho_1}{\rho_0} + \frac{z^2\vep^2}{c^2_1} \langle K^{0}_{D,\Gamma,2} \varphi_0, 1\rangle_{L^2(\Gamma)} + \mathcal R^{(2)}_{\textrm{res},\vep}(z)\right) \varphi_0  = 0, 
\end{align}
where $\mathcal R^{(2)}_{\textrm{res},\vep}$ satisfies
\begin{align*}
\left\|\mathcal R^{(2)}_{\textrm{res},\vep}(z)\right\|_{L^2(\Gamma)} \le C \vep^3.
\end{align*}
In conjunction with the identity (\cite{LS-04}[Lemma 2.6])
\begin{align*}
\frac{1}{8\pi}\int_{\Gamma} \int_{\Gamma} \frac{\nu(x)\cdot (x-y)}{|x-y|} \left(S^{-1}_01\right)(y) d\sigma(y)d\sigma(x) = |\Omega|,
\end{align*}
we arrive at 
\begin{align}\label{eq:148}
\frac{\rho_1\vep^2}{\rho_0} - \frac{z^2 \vep^2} {\mathcal C_\Omega c_1^2}|\Omega| + \mathcal R^{(3)}_{\textrm{res},\vep}(z) = 0,
\end{align}
where $\mathcal R^{(3)}_{\textrm{res},\vep}$ satisfies
\begin{align} \label{eq:149}
\left|\mathcal R^{(3)}_{\textrm{res},\vep}(z)\right| \le C\vep^3.
\end{align}
Dividing by the constant $\omega_M^2|\Omega|\mathcal C^{-1}_\Omega c^{-2}_1$ on both sides of \eqref{eq:148}, we end up with the following characteristic equation for estimating the resonance:
\begin{align}\label{Final-characteristic-equation}
\vep^2 - \frac{\vep^2 z^2}{\omega^2_M} + \frac{R^{(3)}_{\textrm{res},\vep}\mathcal C_\Omega c^2_0}{\omega^2_M|\Omega|} =0.
\end{align}
Note that \(\mathcal R_{\mathrm{res},\varepsilon}^{(3)}\) satisfies
estimate~\eqref{eq:149} and, by statement \eqref{d1} of Lemma~\ref{le:1}, depends analytically on
$z$. Rouché's theorem therefore shows that \eqref{Final-characteristic-equation} admits a unique zero in $\{s\in I: |s-\omega_M^{(3)}|<\delta \}$, which admits the following asymptotic expansion:
\begin{align*}
z(\vep)= \omega^{(3)}_M + z_{res}(\vep)
\end{align*}
with 
\begin{align*}
\left|z_{res}(\vep)\right| \le C\vep, \quad \mathrm{as}\; \vep \rightarrow 0.
\end{align*}
This shows that  the asymptotic expansion \eqref{eq:asy} holds for the case of $d = 3$. 

\medskip
\noindent\textbf{Case 2: Two-dimensional case.}
We now consider the two-dimensional case. The argument follows the similar strategy as in the three-dimensional case. The only essential modifications comes from the logarithmic low-frequency behaviour of the two-dimensional Green function.

In conjunction with \eqref{eq:49} and \eqref{eq:51}, we have the following two-dimensional
analogue of \eqref{eq:143},
\begin{align*}
0 &= \bigg[\frac{\rho_1 \vep^2|\ln \vep|}{\rho_0} \mathbb I + \left(\beta^{(2)}_\vep + \alpha c_0^2 \mathcal P\right)\bigg(\frac 12 +  \mathcal K_{D,\Gamma}^{0,*}  + \mathcal K^{z}_{\vep,\mathrm{bod}} + \frac{(\vep z)^2}{c^2_0} \mathcal K_{D,\Gamma,1, 2}^{0}
\\
& + \frac{(\vep z)^2}{c^2_0}\ln(\vep z/c_0)\mathcal K_{D,\Gamma,2,2}^{0} \bigg) + \mathcal M_\vep^{(1)}(z) + \mathcal M^{(2)}_{\vep}(z) + \mathcal R^{(1)}_{\textrm{res},\vep}(z) \bigg]l_\vep,
\end{align*}
where 
\begin{align*}
    \left\|\mathcal R^{(1)}_{\textrm{res},\vep}\right\|_{\mathcal L\left(L^2(\Gamma)\right)} \le C \vep^3 \ln \vep.
\end{align*}
Here,
\begin{align*}
&\left(\mathcal K_{D,\Gamma,1, 2}^{0}\phi\right)(x)
= {-}\int_{\Gamma} \frac{\partial}{\partial \nu(x)}
\!\left(|x-y|^{2}\big(b_1\ln|x-y|+s_1\big)\right)\phi(y)\,d\sigma(y)\\
&\left(\mathcal K_{D,\Gamma,2, 2}^{0}\phi\right)(x)
= {-}\int_{\Gamma} b_1\,\frac{\partial |x-y|^{2}}{\partial \nu(x)}\,\phi(y)\,d\sigma(y).
\end{align*}
Proceeding as in the derivation of \eqref{eq:46}, and applying Lemma \ref{lem:equilibrium-projection}, we have 
\begin{align}
&\bigg(\frac{\vep^2  |\ln \vep| \rho_1}{\rho_0}\langle\varphi_0,1 \rangle_{L^2(\Gamma)} + \frac{(z \vep)^2}{c^2_1} \langle \mathcal K_{D,\Gamma,1, 2}^{0} \varphi_0, 1\rangle_{L^2(\Gamma)} + \frac{(z \vep)^2}{c^2_1}\ln(z\vep/c_0) \langle \mathcal K_{D,\Gamma,2, 2}^{0}\varphi_0, 1\rangle_{L^2(\Gamma)} \notag\\
& +  \langle K^{z}_{\vep,\mathrm{bod}} \varphi_0, 1\rangle_{L^2(\Gamma)}+ \vep \alpha z^2 \langle SL_{\vep,\mathrm{bod}}^{z}\varphi_0,1\rangle_{L^2(\Omega)}  + \mathcal R^{(2)}_{\textrm{res},\vep}(z)\bigg) = 0, \label{eq:6}
\end{align}
where
\begin{align*}
\left|\mathcal R^{(2)}_{\textrm{res},\vep}\right| \le C\vep^3\ln\vep.
\end{align*}
We note that 
\begin{align*}
\left(\mathcal K_{D,\Gamma,2, 2}^{0}\right)^{*}[1](x)
&= {- 4}\,\overline{b}_1\,|\Omega|;\\
\left(\mathcal K_{D,\Gamma,1, 2}^{0}\right)^{*}[1](x)
& = {-}\left({4}\,\overline{b}_1 + 4\,\overline{s}_1\right)|\Omega|
 \; { -\; 4}\,\overline{b}_1 \int_{\Omega}\ln|x-y|\,dy,
\end{align*}
and that 
\begin{align*}
\left\langle\varphi_0,\;\frac{1}{2\pi}\int_{\Omega}\ln|x-y|\,dy\right\rangle_{L^2(\Gamma)}
= -\gamma_0 |\Omega|.
\end{align*}
Combining the above three identities with  \eqref{eq:5}, \eqref{eq:57} and \eqref{eq:6}, we have 
\begin{align*}
&\frac{\vep^2 |\ln \vep| \rho_1}{\rho_0}=    \frac{(z\vep)^2}{c^2_1} \bigg[\big( {4}\,{b}_1+4\,{s}_1\big)|\Omega| {-} \frac{4b_1 2 \pi \gamma_0|\Omega|}{\langle\phi_0,1 \rangle_{L^2(\Gamma)} } \\ 
& {-} H_D(y_0,y_0;z/c_0)\frac{c_0^2}{c_1^2}|\Omega|\bigg] +  \frac{(z\vep)^2}{c^2_1}\ln(\vep z/c_0)  4\,{b}_1\,|\Omega| - \mathcal R^{(2)}_{\textrm{res},\vep}.
\end{align*}
From this, using Rouché's theorem again, we readily obtain that the interval $\{s\in I: |s-\omega_M^{(2)}|<\delta \}$ contains a unique eigenfrequency admitting the asymptotic expansion \eqref{eq:asy} holds for the case of $d = 2$. 
\end{proof}

\begin{remark}
The leading-order two-dimensional Minnaert frequency is obtained by
retaining only the dominant logarithmic contribution. Since the convergence
rate is only of order $|\ln\varepsilon|^{-1}$, the leading formula may
not be sufficiently accurate for moderate values of $\varepsilon$. For
this reason, in the numerical experiments we can also compare the computed
eigenfrequency with the root of a refined characteristic equation which
keeps the regular $O(1)$ terms in the two-dimensional expansion.

More precisely, set
\begin{align*}
A_M^{(2)}
:=
\big({4}b_1+4s_1\big)|\Omega|
+
\frac{4b_1\,2\pi\gamma_0|\Omega|}
{\langle\phi_0,1\rangle_{L^2(\Gamma)}}
{-}
H_D\left(y_0,y_0;\omega_M^{(2)}/c_0\right) \frac{c_0^2}{c_1^2} |\Omega|,
\end{align*}
where
\begin{align*}
b_1=-\frac{1}{8\pi},
\qquad
s_1=-\frac{1}{8\pi}
\left(\gamma-\ln2-1-\frac{i\pi}{2}\right), \qquad \gamma_0 = {S_\Gamma^0} \left(\widehat {\mathcal S}^0_{\Gamma,k}\right)^{-1} 1.
\end{align*}
We define $z_{\varepsilon,\mathrm{ref}}$ as the root near
$\omega_M^{(2)}$ of
\begin{align}\label{eq:modify}
\frac{\rho_1}{\rho_0}
-
\frac{z^2}{c_1^2|\ln\varepsilon|}
\left[
A_M^{(2)}
+
4b_1|\Omega|\ln\left(\frac{\varepsilon z}{c_0}\right)
\right]
=0 .
\end{align}
In the numerical experiments, we can solve this refined equation and compare
$z_{\varepsilon,\mathrm{ref}}$ with the computed eigenfrequency.
\end{remark}

\subsection{Proof of statement (\ref{a2}) in Theorem 
\ref{thm:main}} \label{sec:p2}

We begin with the following Lippmann-Schwinger equation, whose integral kernel is the free-space Green's function: any solution $u \in H^1_0(D)$ in \eqref{eq:7}--\eqref{eq:8} solves 
\begin{align*}
&\quad u(x) = \alpha z^2 \int_{\Omega_\vep(y_0)} G(x,y;z/c_0) u(y)ds(y)\\
& + \int_{\partial D} \partial_\nu u(y) G(x,y;z/c_0) ds(y) + \int_{\partial \Omega_\vep(y_0)} \left(1 -\frac{\rho_0}{\rho_1\tau^{(d)}_\vep}\right)\partial_\nu u^-(y) G(x,y;z/c_0)ds(y)
\end{align*}
for $x\in D \backslash \partial \Omega_\vep(y_0)$.
Here, $\alpha $ is defined by \eqref{eq:16}.

Based on the above integral expression, it is evident that the field $u$ in $D\backslash \partial \Omega_\vep(y_0)$ can be fully computed using the value $u$ within $\Omega_\vep(y_0)$ and the normal derivatives $\partial_\nu u$ on $\partial \Omega_\vep(y_0)$ and $\partial D$. These three quantities are the solutions of the succeeding integral system
\begin{align*}
& \mathcal A_\vep(z)
\begin{pmatrix}
    u|_{\Omega_\vep(y_0)}\\
    \frac{\rho_0}{\rho_1  \tau^{(d)}_\vep}\partial_{\nu} u^-|_{\partial \Omega_\vep(y_0)}\\
    \partial_{\nu} u|_{\partial D}
\end{pmatrix} = 0,
\end{align*}
where the integral operator $\mathcal A_\vep \in \mathcal L(L^2(\Omega_\vep(y_0)) \times L^2(\partial \Omega_\vep(y_0)) \times L^2(\partial D))$ is defined by
\begin{align*}
\mathcal A_\vep(z):=
  \begin{bmatrix}&\mathbb I - z^2 \alpha \mathcal N^{z/c_0}_{\Omega_\vep(y_0)} &  \left(1 - \frac{\rho_1  \tau^{(d)}_\vep}{\rho_0}\right) \mathcal{SL}^{z/c_0}_{\partial \Omega_\vep(y_0)} & - \mathcal{SL}^{z/c_0}_{\partial D}|_{\Omega_\vep(y_0)} \\
& - z^2 \alpha \partial_\nu \mathcal N^{z/c_0}_{\Omega_\vep(y_0)} & \frac{\mathbb I} 2 + \mathcal K^{z/c_0,*}_{\partial \Omega_\vep(y_0)}  + \frac{\rho_1  \tau^{(d)}_\vep}{\rho_0}\left(\frac{\mathbb I} 2 - \mathcal K^{z/c_0,*}_{\partial \Omega_\vep(y_0)}\right) & -\partial_\nu \mathcal{SL}^{z/c_0}_{\partial D}|_{\partial \Omega_\vep(y_0)}\\
&-\partial_\nu \alpha \lambda^2  \mathcal{N}^{z/c_0}_{\Omega_\vep(y_0)}|_{\partial D}& \left(1 - \frac{\rho_1  \tau^{(d)}_\vep}{\rho_0}\right) \partial_\nu {\mathcal{SL}^{z/c_0}_{\partial \Omega_\vep(y_0)}}|_{\partial D} & \frac{\mathbb I} 2 - \mathcal K_{\partial D}^{z/c_0,*}
\end{bmatrix}.
\end{align*}

Throughout this section, we consider the part of the spectrum $\Sigma^{(d)}_\vep$ associated with background Dirichlet frequencies away from the Minnaert
frequency. More precisely, we choose $\delta>0$ such that the closed balls ${B(z_0,\delta)}:=\{z\in \mathbb C:|z-z_0|< \delta\}$, with $z_0\in\Lambda_D\cap I$, together with $\overline{B(\omega_M^{(d)},\delta)}$ are pairwise disjoint. This guarantees that \eqref{eq:15} holds.

Before proving statement \eqref{a2} of Theorem \ref{thm:main}, we prepare two technical lemmas that will be used in the proof.

\begin{lemma}
\label{lem:local-factorization}
Let $\vep >0$ be fixed and let $z_0\in\Sigma_\varepsilon^{(d)}$ be a characteristic value of
$\mathcal A_\varepsilon(z)$. Set
\begin{align*} 
n_{z_0}:=\dim\ker \mathcal A_\varepsilon(z_0).
\end{align*}
There exists a neighbourhood $U$ of $z_0$ such that 
one has, for $z\in U$,
\begin{align}\label{eq:local-factorization}
\mathcal A_\varepsilon(z)
= \mathcal E_\varepsilon(z)
    \left[ \mathbb I-\sum_{j=1}^{n_{z_0}}\mathcal P_j(z_0)
        + (z-z_0)\sum_{j=1}^{n_{z_0}}\mathcal P_j(z_0) \right]
    \mathcal F_\varepsilon(z).
\end{align}
Here 
$\mathcal E_\varepsilon(z),\ \mathcal F_\varepsilon(z)$ are two holomorphic
operator-valued functions, which are boundedly invertible for every $z\in U$, $ 
    \mathcal P_1(z_0),\ldots,\mathcal P_{n_{z_0}}(z_0)$ are 
mutually disjoint rank-one
projections. 
\end{lemma}

\begin{proof}
We first claim that if
\begin{align} \label{eq:19}
   \mathcal A_\vep(z) \begin{bmatrix}
       \phi_1(z)\\
       \phi_2(z)\\
       \phi_3(z)
   \end{bmatrix} = (z-z_0)^m
   \begin{bmatrix}
       \widetilde \phi_1(z)\\
       \widetilde \phi_2(z)\\
       \widetilde \phi_3(z)
   \end{bmatrix},
   \end{align}
then we have $m=1$. Here, $(\phi_1(z),\phi_2(z),\phi_3(z))$ and $\left(\widetilde \phi_1(z), \widetilde \phi_2(z), \widetilde \phi_3(z)\right)$ are holomorphic near $z_0$ and 
\begin{align}\label{eq:32}
\left(\phi_1(z_0),\phi_2(z_0), \phi_3(z_0) \right) \ne \textbf{0},\qquad \left(\widetilde \phi_1(z_0),\widetilde \phi_2(z_0), \widetilde \phi_3(z_0) \right) \ne \textbf{0}.
\end{align}
Define 
\begin{align*}
 \psi(z)=  z^2 \alpha N^{z/c_0}_{\Omega_\vep(y_0)} \phi_1(z) - \left(1 - \frac{\rho_1  \tau^{(d)}_\vep}{\rho_0}\right) SL^{z/c_0}_{\partial \Omega_\vep(y_0)}\phi_2(z) +  SL^{z/c_0}_{\partial D}\phi_3(z).
\end{align*}
By a straightforward calculation, we have 
\begin{align}
&\Delta \psi(z) + z^2/c^2_0\psi(z) = -z^2\alpha \phi_1 (z) \quad \mathrm{in}\; \Omega_\vep(y_0). \label{eq:re_5}
\end{align}
Using \eqref{eq:19}, we obtain 
\begin{align}
&\phi_1(z) - \psi(z) = (z-z_0)^m  \widetilde \phi_1(z)\quad \mathrm{in}\; \Omega_\vep(y_0), \label{eq:12}\\
& \Delta \psi(z) + \frac{z^2}{c^2_0}  \psi(z) = 0 \quad  \mathrm{in}\; D \backslash \overline{\Omega_\vep(y_0)},\label{eq:22}\\
& \partial^-_\nu \psi_+(z) = \frac{\rho_1 \tau_\vep^{(d)}}{\rho_0}\partial^+_\nu \psi_-(z)- \left(1 - \frac{\rho_1 \tau_\vep^{(d)}}{\rho_0}\right) (z-z_0)^m \widetilde \phi_2(z) \quad \mathrm{on}\; \partial \Omega_\vep(y_0), \label{eq:20}
\end{align}
and
\begin{align}
\psi_+(z) = \psi_-(z),\quad \partial^+_\nu \psi(z) = - (z-z_0)^m \widetilde \phi_3(z) \quad \mathrm{on}\; \partial D.\label{eq:21}
\end{align}
In conjunction with \eqref{eq:21} and the well-posedness of the external scattering problem, we have 
\begin{align} \label{eq:24}
\|\gamma \psi(z)\|_{H^{1/2}(\Gamma)} \le C |z-z_0|^m.
\end{align}
Combining \eqref{eq:12} and \eqref{eq:re_5} gives 
\begin{align} \label{eq:23}
\Delta \psi(z) + \frac{z^2}{c^2_1}\psi(z) = -(z-z_0)^m \alpha \widetilde \phi_1(z) \quad \mathrm{in}\; \Omega_\vep(y_0).
\end{align}
Furthermore, setting $z = z_0$ in  \eqref{eq:22}, \eqref{eq:20}, \eqref{eq:24} and \eqref{eq:23}, we obtain  
\begin{align*}
&\Delta \psi(z_0) + \frac{z_0^2}{c^2_0}  \psi(z_0) = 0 \qquad\qquad \ \mathrm{in}\; D \backslash \overline{\Omega_\vep(y_0)},\\
&\Delta \psi(z_0) + \frac{z_0^2}{c^2_1} \psi(z_0) = 0 \qquad\qquad\;\;\mathrm{in}\;\overline{\Omega_\vep(y_0)},\\
&\partial^-_\nu \psi_+(z_0) = \frac{\rho_1 \tau_\vep^{(d)}}{\rho_0}\partial^+_\nu \psi_-(z_0) \qquad \mathrm{on}\; \partial \Omega_\vep(y_0), \\
&\psi(z_0) = 0 \qquad \qquad\qquad\qquad\qquad\;\; \mathrm{on}\; \partial D,
\end{align*}
From this, using \eqref{eq:22}, \eqref{eq:20},  and \eqref{eq:23}, and applying Green formulas, we obtain 
\begin{align}
&\int_{\partial D}[\partial_\nu \psi(z) \overline{\psi(z_0)} - \psi(z) \partial_\nu \overline{\psi(z_0)}](y) d\sigma(y) \notag \\
&= \int_{D\backslash \overline{\Omega_\vep(y_0)}}  
\frac{z_0^2-z^2}{c_0^2}
\left[\psi(z)\overline{\psi(z_0)}\right](y)\,dy + \int_{\partial \Omega_\vep(y_0)}[\partial^+_\nu \psi(z) \overline{\psi(z_0)} - \psi(z) \partial^+_\nu \overline{\psi(z_0)}](y) d\sigma(y) \notag \\
& = \int_{D\backslash \overline{\Omega_\vep(y_0)}}  
\frac{z_0^2-z^2}{c_0^2}
\left[\psi(z)\overline{\psi(z_0)}\right](y)\,dy \notag \\
&\qquad + \frac{\rho_1 \tau_\vep^{(d)}}{\rho_0}\int_{\partial \Omega_\vep(y_0)}[\partial^-_\nu \psi(z) \overline{\psi(z_0)} - \psi(z) \partial^-_\nu \overline{\psi(z_0)}](y) d\sigma(y) \notag \\
& = \int_{D\backslash \overline{\Omega_\vep(y_0)}}  
\frac{z_0^2-z^2}{c_0^2}
\left[\psi(z)\overline{\psi(z_0)}\right](y)\,dy + \frac{\rho_1 \tau_\vep^{(d)}}{\rho_0}
\int_{\Omega_\vep(y_0)}\frac{z_0^2-z^2}{c_1^2}
\left[\psi(z)\overline{\psi(z_0)}\right](y)\,dy
\notag \\
&\qquad + \frac{\rho_1 \tau_\vep^{(d)}}{\rho_0} \int_{\Omega_\vep(y_0)} (z-z_0)^m z^2\alpha
\widetilde \phi_1(z)\overline{\psi(z_0)}\,dy. \label{eq:42} 
\end{align}
If $m\ge2$, dividing by $z-z_0$ and letting $z\to z_0$, we deduce from \eqref{eq:42} that  
\begin{align*}
-\frac{2z_0}{c_0^2} \int_{D\backslash \overline{\Omega_\vep(y_0)}}|\psi(z_0)|^2(y)\,dy -\frac{2z_0}{c_1^2} \frac{\rho_1 \tau_\vep^{(d)}}{\rho_0} \int_{\Omega_\vep(y_0)} |\psi(z_0)|^2(y)\,dy=0,
\end{align*}
which implies 
\begin{align*}
\psi(z_0)=0 \qquad \mathrm{in}\; D.
\end{align*}
From this, we obtain $\left(\phi_1(z_0),\phi_2(z_0), \phi_3(z_0) \right)= 0$, which is a contradiction to the nonzero assumption \eqref{eq:32}. Therefore, we have $m =1 $.

Finally, we note that $\mathcal A_\vep(z)$ is a Fredhlom operator with index 0 and $\mathcal A_0$ is normal at its zeros. 
Then, for each zero $z_0$ of $\mathcal A_\vep(z)$, by using Gohberg and Sigal theory (cf.\cite{{AK-09}}), we obtain
\begin{align*}
\mathcal A_\varepsilon(z)
= \mathcal E_\varepsilon(z)
    \left[ \mathbb I-\sum_{j=1}^{n_{z_0}}\mathcal P_j(z_0)
        + (z-z_0)^{k_l(z_0)}\sum_{j=1}^{n_{z_0}}\mathcal P_j(z_0) \right]
    \mathcal F_\varepsilon(z).
\end{align*}
where $\mathcal P_1(z_0),\ldots, \mathcal P_{n_{0}}(z_0)$ are mutually disjoint one-dimensional projections, 
\begin{align*}
k_1(z_0)\ldots, k_{n_{z_0}}(z_0)\in \mathbb N,
\end{align*}
and $\mathcal F (z)$ and $\mathcal E(z)$ are both holomorphic and invertible near $z_0$. This, together with the previous proved claim shows that 
\begin{align*}
    k_1(z_0)=\cdots= k_{n_{z_0}}(z_0) = 1.
\end{align*}
Thus the proof of this lemma is thus completed.
\end{proof}

Define
\begin{align*}
\mathcal A_\vep^{sc}(z):= \begin{bmatrix}&\mathbb I - z^2 \alpha \mathcal N^{z/c_0}_{\Omega_\vep(y_0)} & \left(1 - \frac{\rho_1  \tau^{(d)}_\vep}{\rho_0}\right) \mathcal {SL}^{z/c_0}_{\partial \Omega_\vep(y_0)} \\
& - z^2 \alpha \partial_\nu \mathcal {N}^{z/c_0}_{\Omega_\vep(y_0)} & \frac{\mathbb I} 2 + \mathcal K^{z/c_0,*}_{\partial \Omega_\vep(y_0)}  + \frac{\rho_1  \tau^{(d)}_\vep}{\rho_0}\left(\frac{\mathbb I} 2 - \mathcal K^{z/c_0,*}_{\partial \Omega_\vep(y_0)}\right)
\end{bmatrix}.
\end{align*}
The invertibility of $\mathcal A_\vep^{sc}(z)$ plays a key role in the proof of \eqref{a2} of Theorem \ref{thm:main}. To this end, we introduce the following unitary operators:
\begin{align*}
(\mathbb U_{\varepsilon,\Omega}f)(x)
&:=\varepsilon^{d/2}f(\Phi_\varepsilon(x;y_0)),
&& f\in L^2(\Omega_\varepsilon(y_0)),\quad x\in\Omega,\\
(\mathbb U_{\varepsilon,\Gamma}g)(x)
&:=\varepsilon^{(d-1)/2}g(\Phi_\varepsilon(x;y_0)),
&& g\in L^2(\partial \Omega_\varepsilon(y_0)),\quad x\in\Gamma.
\end{align*}
and consider the corresponding unitarily  operator:
\begin{align*}
\widetilde{\mathcal A}_\varepsilon^{\mathrm{sc}}(z)
:=
\operatorname{diag}
\bigl(\mathbb U_{\varepsilon,\Omega}, \mathbb U_{\varepsilon,\Gamma}\bigr)
\mathcal A_\varepsilon^{\mathrm{sc}}(z)
\left(\operatorname{diag}
\bigl(\mathbb U_{\varepsilon,\Omega}, \mathbb U_{\varepsilon,\Gamma}\bigr)\right)^{-1}.
\end{align*}
Clearly, whenever the inverses exist,
\begin{align*}
\left\|
\bigl(\widetilde{\mathcal A}_\varepsilon^{\mathrm{sc}}(z)\bigr)^{-1}
\right\|_{\mathcal L(L^2(\Omega)\times L^2(\Gamma))}
=
\left\|
\bigl(\mathcal A_\varepsilon^{\mathrm{sc}}(z)\bigr)^{-1}
\right\|_{\mathcal L(L^2(\Omega_\varepsilon)\times
L^2(\Gamma_\varepsilon))}.
\end{align*}
The following lemma provides the required invertibility estimate for the unitarily equivalent operator.

\begin{lemma}\label{lem:scattering-resonance}
Let $I\Subset(0,\infty)$ be a bounded interval separated from
the Minnaert frequency $\omega^{(d)}_M$. Then there exist constants
$\varepsilon_0>0$, $ 0 < \delta_0 <\delta$ and $C_{\delta_0} >0$, independent of $\varepsilon$, such that
for all $0<\varepsilon<\varepsilon_0$ and all 
\begin{align*}
z\in B^{\mathbb C}_{I,\delta_0}:=\{z\in\mathbb C: |z-\omega^{(d)}_M| > \delta_0,\; \mathrm{Re}(z) \in I, \;|\mathrm{Im}(z)| < \delta_0\},
\end{align*}
the inverse of $\widetilde{\mathcal A}_{\varepsilon}^{\mathrm{sc}}(z)$ exists.

Precisely,  given $\mathbf g=(g_\Omega,g_\Gamma)$, let 
\begin{align*}
(\widetilde g_\Omega(z),\widetilde g_\Gamma(z))
=
\bigl(\widetilde{\mathcal A}_{\varepsilon}^{\mathrm{sc}}(z)\bigr)^{-1}
\mathbf g,
\end{align*}
then, for $z \in B^{\mathbb C}_{I,\delta_0}$, we have 
\begin{align} \label{eq:26}
\|\widetilde g_\Omega(z)\|_{L^2(\Omega)} + \| \widetilde g_\Gamma(z)\|_{L^2(\Gamma)}
\leq
C_{\delta_0}
\left(
\frac{\varepsilon^{1/2}}{\tau_\varepsilon^{(d)}}
\|g_\Omega\|_{L^2(\Omega)}
+
\|\mathcal Qg_\Gamma\|_{L^2(\Gamma)}
+
\frac{1}{\tau_\varepsilon^{(d)}}
\|\mathcal Pg_\Gamma\|_{L^2(\Gamma)}
\right).
\end{align}
\end{lemma}

\begin{proof}

Using the scaling relations for the integral operators, we obtain
\begin{align*}
\widetilde{\mathcal A}_\varepsilon^{\mathrm{sc}}(z)
=
\begin{pmatrix}
\mathbb I - \alpha (\vep z)^2\mathcal N_\Omega^{\vep z/c_0}
&
\left(1- \frac{\rho_1  \tau^{(d)}_\vep}{\rho_0}\right)\varepsilon^{3/2}
\mathcal{SL}_\Gamma^{\vep z/c_0}
\\[1mm]
-\alpha z^2\varepsilon^{1/2}
\partial_\nu\mathcal N_\Omega^{\vep z/c_0}
&
\displaystyle
\frac{\mathbb I}{2}
+\mathcal K_\Gamma^{\vep z/c_0,*}
+\frac{\rho_1  \tau^{(d)}_\vep}{\rho_0}
\left(
\frac{\mathbb I}{2}
-\mathcal K_\Gamma^{\vep z/c_0,*}
\right)
\end{pmatrix}=: 
\begin{pmatrix}
X_\varepsilon(z) & Y_\varepsilon(z)\\
Z_\varepsilon(z) & W_\varepsilon(z)
\end{pmatrix}.
\end{align*}
The upper-left block is uniformly invertible for all sufficiently
small $\vep$. Therefore, the invertiblility of
$\widetilde{\mathcal A}_\varepsilon^{\mathrm{sc}}(z)$ is equivalent
to that of the following Schur complement on $L^2(\Gamma)$:
\begin{align*}
\begin{aligned}
\mathscr B^{\mathrm{sc}}(z)
:=
&\frac{\mathbb I}{2}
+\mathcal K_\Gamma^{\vep z/c_0,*}
+\frac{\rho_1  \tau^{(d)}_\vep}{\rho_0}
\left(
\frac{\mathbb I}{2}
-\mathcal K_\Gamma^{\vep z/c_0,*}
\right)
\\
&\quad
+\alpha\left(1 -\frac{\rho_1  \tau^{(d)}_\vep}{\rho_0}\right)(\vep z)^2
\partial_\nu\mathcal N_\Omega^{\vep z/c_0} \left(
\mathbb I-\alpha (\vep z)^2\mathcal N_\Omega^{\vep z/c_0}
\right)^{-1}
\mathcal{SL}_\Gamma^{\vep z/c_0}.
\end{aligned}
\end{align*}
The desired estimate then follows by repeating the
$\mathcal P$-$\mathcal Q$ decomposition as in the derivation of \eqref{eq:45}--\eqref{eq:46} and the low-frequency asymptotics of those operators, we obtain that 
\begin{align} \label{eq:53}
\left\|
\bigl(\mathscr B_{\varepsilon}^{\mathrm{sc}}(z)\bigr)^{-1}h
\right\|_{L^2(\Gamma)}
\leq
C_{\delta_0}
\left(
\left\|\mathcal Qh\right\|_{L^2(\Gamma)}
+
\frac{1}{\tau_\varepsilon^{(d)}}
\left\|\mathcal Ph\right\|_{L^2(\Gamma)}
\right), \quad \mathrm{for}\; z\in B^{\mathbb C}_{I,\delta_0}.
\end{align}
Furthermore,
block elimination yields
\begin{align*}
&\widetilde g_\Gamma(z)
= 
\bigl(\mathscr B_{\varepsilon}^{\mathrm{sc}}(z)\bigr)^{-1}
\left(
g_\Gamma-Z_\varepsilon(z)X_\varepsilon(z)^{-1}g_\Omega
\right)\\
\mathrm{and}\qquad
&\widetilde g_\Omega(z)
=
X_\varepsilon(z)^{-1}
\left(
g_\Omega-Y_\varepsilon(z)\widetilde g_\Gamma(z)
\right).
\end{align*}
This, together with \eqref{eq:53} yields \eqref{eq:26}.
\end{proof}

Combining Lemmas \ref{lem:local-factorization} and \ref{lem:scattering-resonance}, we are ready to prove statement (\ref{a2}) of Theorem \ref{thm:main}.

\begin{proof}[Proof of statement \eqref{a2} of Theorem \ref{thm:main}]
We denote the Schur complement of $\mathcal A_\vep^{\mathrm{sc}}$ by 
\begin{align*}
\mathcal A_\vep^{\mathrm{sc},\mathrm{shur}}(z):=
\frac{\mathbb I} 2 - \mathcal K_{\partial D}^{z/c_0,*} -  \mathcal A^{\mathrm{sc},\mathrm{rem}}_\vep(z),
\end{align*}
where
\begin{align*}
\mathcal A^{\mathrm{sc},\mathrm{rem}}_\vep(z)&:= \begin{bmatrix}
    &-\partial_\nu \alpha \lambda^2  \mathcal{N}^{z/c_0}_{\Omega_\vep(y_0)}|_{\partial D}& \left(1 - \frac{\rho_1  \tau^{(d)}_\vep}{\rho_0}\right) \partial_\nu {\mathcal{SL}^{z/c_0}_{\partial \Omega_\vep(y_0)}}|_{\partial D}
\end{bmatrix} \\
&\left(\mathcal A^{\mathrm{sc}}_{\vep}(z)\right)^{-1}  \begin{bmatrix}
\mathcal {SL}^{z/c_0}_{\partial D}|_{\Omega_\vep(y_0)} \\
\partial_\nu \mathcal{SL}^{z/c_0}_{\partial D}|_{\partial \Omega_\vep(y_0)}
\end{bmatrix}.
\end{align*}
Since the invertibility of $\mathcal A_\vep$ reduces to the invertibility of $\mathcal A^{\mathrm{sc},\mathrm{shu r}}_\vep$, it suffices to investigate the zeros of $\mathcal A^{\mathrm{sc},\mathrm{shur}}_\vep$. 
It is easy to verify that for $z \in \overline{B(z_0,\delta)}$ with $z_0 \in \Lambda_D \cap I$,
\begin{align}
&\left\|\partial_\nu \alpha \lambda^2  \mathcal N^{z/c_0}_{\Omega_\vep(y_0)}\right\|_{_{\mathcal L(L^2(\Omega_\vep(y_0)),L^2(\partial D))}} + \left\|\partial_\nu {\mathcal {SL}^{z/c_0}_{\partial \Omega_\vep(y_0)}}\right\|_{\mathcal L(L^2(\partial\Omega_\vep(y_0)),L^2(\partial D))}\le C\vep^{\frac{d-1}{2}}. \label{eq:27}
\end{align}
Furthermore, by a straightforward calculation, we have: for $g\in L^2(\partial D)$,
\begin{align}
&\left\| \mathbb U_{\vep,\Omega}  \left(\mathcal{SL}^{z/c_0}_{\partial D} g|_{\Omega_\vep(y_0)}\right)\right\|_{L^2(\Omega)} \le C \vep^{\frac{d}2}\|g\|_{L^2(\partial D)}, \notag \\
& \left\|\mathcal Q\mathbb U_{\vep,\Gamma}\left(\partial_\nu \mathcal{SL}^{z/c_0}_{\partial D}g|_{\partial \Omega_\vep(y_0)}\right) \right\|_{L^2(\Gamma)} \le C \vep^{\frac{d-1}2}\|g\|_{L^2(\partial D)}, \notag\\
&\left\|\mathcal P\mathbb U_{\vep,\Gamma}\left(\partial_\nu \mathcal{SL}^{z/c_0}_{\partial D}g|_{\partial \Omega_\vep(y_0)}\right) \right\|_{L^2(\Gamma)} \le C \vep^{\frac{d+1}2}\|g\|_{L^2(\partial D)}. \notag 
\end{align}
This, together with Lemma \ref{lem:scattering-resonance} and \eqref{eq:27} gives
\begin{align} \label{eq:29}
\left\|\mathcal A_\vep^{\mathrm{sc},\mathrm{rem}}(z)\right\|_{\mathcal L(L^2(\partial D))} \le C_0 \eta_d(\vep), \quad \mathrm{for}\; B^{\mathbb C}_{I,\delta_0}.
\end{align}

We note that
the zeros of $ \mathbb I/ 2 - \mathcal K_{\partial D}^{z/c_0,*}
$ are precisely the elements of $\Lambda_D$ (see, e.g., \cite[Lemma~11.8]{AK-04}, together with the converse implication following from Green's representation formula). Furthermore, this operator is a Fredhlom of index zero and is normal at each of its zeros (in the sense of Gohberg and Sigal; see
\cite[Section~1.1.4]{AK-09}). 
Moreover, it is known that each $z_0 \in \Lambda_D$ is an isolated semisimple zero of
$\mathbb I/2 -\mathcal K_{\partial D}^{z/c_0,*}$. Hence, after possibly shrinking
$0 < \widetilde \delta < \delta_0$, this operator
is invertible for $0<|z-z_0|< \widetilde \delta$, and there exists a constant
$C_1>0$ independent of $\widetilde \delta$ such that
\begin{align} \label{eq:30}
\left\|
\left(
\frac12\mathbb I-\mathcal K_{\partial D}^{z/c_0,*}
\right)^{-1}
\right\|_{\mathcal L(L^2(\partial D))}
\le
C_1|z-z_0|^{-1},
\qquad
0<|z-z_0|< \widetilde \delta .
\end{align}
Choose $M > C_0 C_1$, independently of $\vep$, and set
\begin{align*}
r_{\vep}:= M \eta_d(\vep).
\end{align*}
For all sufficiently small $\varepsilon>0$, we have
$r_\varepsilon< \widetilde \delta$. Moreover, applying \eqref{eq:29} and \eqref{eq:30}, we have : for every $z\in {\partial B(z_0,r_\vep)}$, 
\begin{align} \label{eq:33}
\left\|
\left(
\frac12\mathbb I-\mathcal K_{\partial D}^{z/c_0,*}
\right)^{-1}A_\vep^{\mathrm{sc},\mathrm{rem}}
\right\|
&\leq
\frac{C_0}{r_\vep}
C_1\eta_d(\vep)
<1.
\end{align}
Therefore, using generalized Rouché theorem (see, e.g., \cite{AK-09}), we obtain that
$\mathbb I/2 -\mathcal K_{\partial D}^{z/c_0,*}$ and $\mathcal A^{\mathrm{sc},\mathrm{shu r}}_\vep$ have the same total
number of zeros inside $ B(z_0,r_\vep)$. Recall that $z_0$ is the unique semisimple
zero of $\mathbb I/2 -\mathcal K_{\partial D}^{z/c_0,*}$ in
$B(z_0,r_\vep)$. Consequently,
\begin{align*}
\mathrm{dim}\left(\mathrm{Ker}\big(\mathbb I/2 -\mathcal K_{\partial D}^{z_0/c_0,*}\big)\right) = m_D(z_0).
\end{align*}
Furthermore, Lemma \ref{lem:local-factorization} implies that the zeros of $\mathcal A^{\mathrm{sc},\mathrm{shu r}}_\vep$ are also
semisimple. It follows that $\mathcal A_{\vep}^{\mathrm{sc},\mathrm{shu r}}$ has precisely $m_D(z_0)$ zeros in $B(z_0,r_\vep)$, denoted by
$z_{\vep,1},\ldots,z_{\vep,m_D(z_0)}$, repeated according to algebraic multiplicity.
In particular,
\begin{align*}
\max_{1\leq j\leq m_D(z_0)}
|z_{\varepsilon,j}-z_0|
\leq M\eta_d(\vep),
\end{align*}
which implies \eqref{eq:31}.

On the other hand, arguing as in the derivation of \eqref{eq:33}, we find that \eqref{eq:33} also holds in the domain
\begin{align*}
\left\{z\in \mathbb C: \mathrm{Re}(z) \in I, \; |\mathrm{Im}(z)| < \delta, \; z \notin B(\omega_M^{(d)},\delta)\right\} \backslash\left( \bigcup_{z_0\in I \cap \Lambda_D} {B(z_0,r_\vep)}\right).
\end{align*}
Hence, the Neumann-series argument shows that $A_\vep^{\mathrm{sc},\mathrm{shur}}$ is invertible. Consequently, there are no eigenfrequencies in this region.
\end{proof}

\section{Numerical experiments}\label{sec:numerics}

In this section, we present numerical experiments to verify the eigenvalue asymptotics in Theorem \ref{thm:main}.
 
Recall that $D$ and $\Omega_\vep(y_0)$ are defined by \eqref{eq:def Omega_e}, where $\vep=10^{-4}$ is fixed throughout of this section. Furthermore, the  reference inclusion $\Omega$ associated with $\Omega_\vep(y_0)$ is chosen as $B_1(0)$ in this section, where $B_1(0)$ is a unit disk centered at the origin in $\mathbb R^2$. The acoustic parameters in \eqref{eq:17} are chosen as follows
\begin{align*}
\rho_0 = \rho_1 = 1, \quad k_0=k_1=1,
\end{align*}
and
\begin{align*}
\rho_\vep(x)=
\begin{cases}
\rho_0, & x\in D\setminus\Omega_\vep(y_0),\\
\rho_1\tau_\vep, & x\in\Omega_\vep(y_0),
\end{cases}
\qquad
k_\vep(x)=
\begin{cases}
k_0, & x\in D\setminus\Omega_\vep(y_0),\\
k_1\tau_\vep, & x\in\Omega_\vep(y_0),
\end{cases}
\end{align*}
with the two-dimensional scaling parameter
\begin{align*}
\tau_\vep=\vep^2|\ln\vep|\approx 9.2103\times10^{-8}.
\end{align*}
The variational formulation of the eigenvalue problem \eqref{eq:7}-\eqref{eq:8} reads: 
find $(\lambda,u)\in\R\times H_0^1(D)$, $\|u\|_{L^2(D)}=1$, such that
\begin{equation}\label{eq:var}
\int_D\frac{1}{\rho_\vep}\nabla u\cdot\nabla v\,dx
=
\lambda\int_D\frac{1}{k_\vep}uv\,dx
\qquad\forall\,v\in H_0^1(D).
\end{equation}
We set the signed frequency $z=\sqrt{\lambda}$.


The computational domain is meshed using \textsc{Gmsh}~\cite{GeuzaineRemacle2009}. Equation \eqref{eq:var} is discretized using conforming $P_1$ Lagrange finite elements. This leads to the generalized matrix eigenvalue problem
\begin{align}\label{eq:dis eig}
K_h U_h=\lambda_h M_h U_h,
\end{align}
where $K_h$ and $M_h$ are the standard stiffness and mass matrices with element-wise coefficients $1/\rho_\vep$ and $1/k_\vep$, respectively. The outer Dirichlet boundary condition is enforced by removing the corresponding degrees of freedom. The eigenpairs are computed with the MATLAB function \texttt{eigs} applied to $(K_h,M_h)$. 

For a centered circular inclusion we have $|\Omega|=\pi$ and hence the two-dimensional leading-order Minnaert prediction
\begin{align*}
\omega_M^{(2)}=\sqrt{\frac{2\pi k_1}{|\Omega|\rho_0}}=\sqrt{2},
\qquad
\lambda_M=\bigl(\omega_M^{(2)}\bigr)^2=2.
\end{align*}

\subsection{Unit disk with a centered inclusion}\label{sec:num-disk}

In the first example, we take the two-dimensional domain $D=B(0,1)$ and $\Omega_\vep(y_0)=B(0,\vep)$. The exact background eigenvalues of the homogeneous disk are $\lambda_{m,n}=j_{m,n}^2$, where $j_{m,n}$ is the $n$-th positive zero of the Bessel function $J_m$; eigenvalues with $m\ge1$ have multiplicity two.
\begin{figure}[!h]
	\centering
	\includegraphics[width=0.5\textwidth]{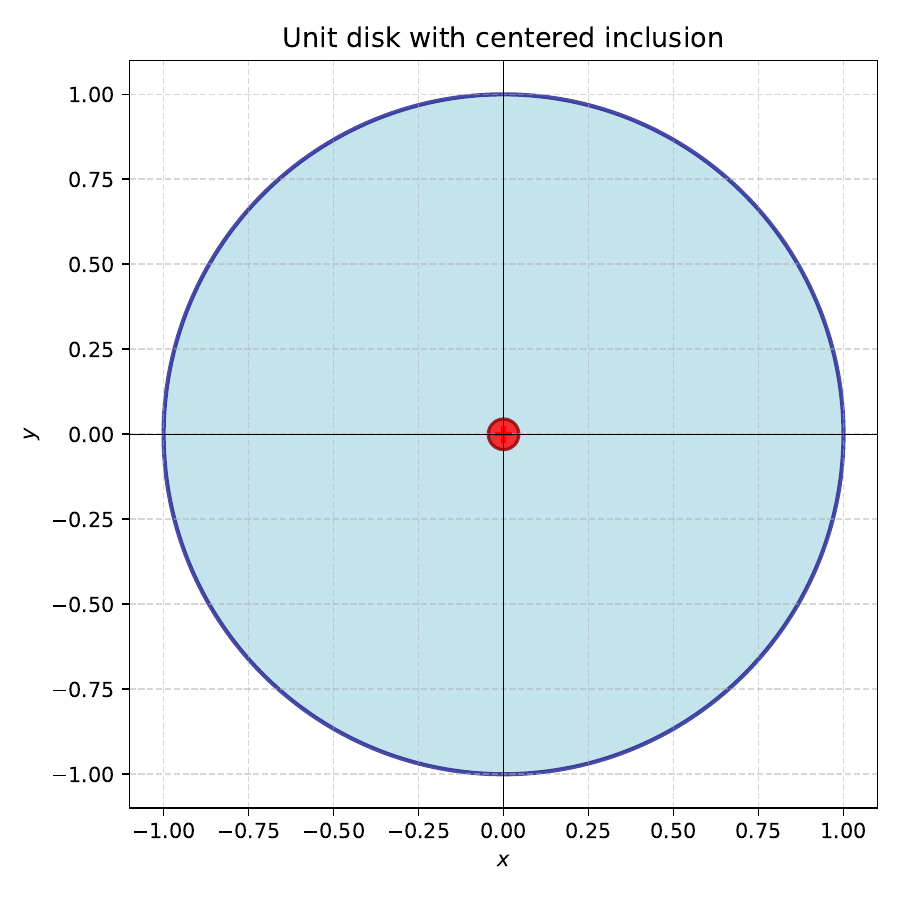}
	\caption{Unit disk with centered inclusion.}
\end{figure}

\begin{table}[!htbp]
	\centering
	\caption{Unit disk with a centered inclusion of size $\vep=10^{-4}$: first ten computed eigenvalues and reference values.}
	\label{tab:disk}
	\small\setlength{\tabcolsep}{4pt}
	\begin{tabular}{cccccc}
		\toprule
		\multicolumn{3}{c}{Eigenvalue $\lambda$} & \multicolumn{3}{c}{Frequency $z=\sqrt{\lambda}$} \\
		\midrule
		Numerical & Reference & Error & Numerical & Reference & Error \\
		\midrule
		1.8633798542  & 1.8441979339   & 1.92$\times 10^{-2}$ & 1.3650567220  & 1.3580124940   & 7.04$\times 10^{-3}$ \\
		7.1125772200  & 5.7831859629  & 1.33$\times 10^{0}$ & 2.6669415479  & 2.4048255577  & 2.62$\times 10^{-1}$ \\
		14.6821549558  & 14.6819706421  & 1.84$\times 10^{-4}$ & 3.8317300213  & 3.8317059702  & 2.41$\times 10^{-5}$ \\
		14.6821558149  & 14.6819706421  & 1.85$\times 10^{-4}$ & 3.8317301334  & 3.8317059702  & 2.42$\times 10^{-5}$ \\
		26.3751981163  & 26.3746164272  & 5.82$\times 10^{-4}$ & 5.1356789343  & 5.1356223018  & 5.66$\times 10^{-5}$ \\
		26.3751994892  & 26.3746164272  & 5.83$\times 10^{-4}$ & 5.1356790680  & 5.1356223018  & 5.68$\times 10^{-5}$ \\
		32.9416549851  & 30.4712623437  & 2.47$\times 10^{0}$ & 5.7394821182  & 5.5200781103  & 2.19$\times 10^{-1}$ \\
		40.7079107802  & 40.7064658182  & 1.44$\times 10^{-3}$ & 6.3802751336  & 6.3801618959  & 1.13$\times 10^{-4}$ \\
		40.7079152332  & 40.7064658182  & 1.45$\times 10^{-3}$ & 6.3802754825  & 6.3801618959  & 1.14$\times 10^{-4}$ \\
		49.2201222737  & 49.2184563217  & 1.67$\times 10^{-3}$ & 7.0157054010  & 7.0155866698  & 1.19$\times 10^{-4}$ \\
		\bottomrule
	\end{tabular}
\end{table}

Table~\ref{tab:disk} lists the first ten computed eigenvalues of \eqref{eq:dis eig}. The first eigenvalue corresponds to the Minnaert branch and is close to 
the characterization \eqref{eq:asy}  in Theorem \ref{thm:main}. The error is 
\begin{align*}
|z_{1,\mathrm{num}}-\sqrt{2}|\approx 4.92\times 10^{-2}. 
\end{align*}
The magnitude of error agrees with the estimation \eqref{eq:asy} with $\eta_2(\vep)\approx 0.10857$.

By solving the refined equation \eqref{eq:modify} numerically, we have the first reference value for the Minnaert frequency $z_{1,\mathrm{ref}}\approx 1.3580124940$. The error is
\begin{align*}
|z_{1,\mathrm{num}}-z_{1,\mathrm{ref}}|\approx 7.044\times 10^{-3}.
\end{align*}

The remaining reference eigenvalues are background Dirichlet modes.  The double eigenvalues near $14.6820$, $26.3746$, and $40.7065$ are reproduced as close pairs, in agreement with the multiplicity preservation stated in Theorem~\ref{thm:main}.

{The two comparatively large discrepancies in Table~\ref{tab:disk} correspond to the radially symmetric modes and remain consistent with the worst-case logarithmic estimate established in Theorem \ref{thm:main}.}

Then we change the size of inclusion to $\vep=10^{-6}$. The exact background eigenvalues are still the same.
\begin{table}[!htbp]
	\centering
	\caption{Unit disk with a centered inclusion of size $\vep=10^{-6}$: first ten computed eigenvalues and reference values.}
	\label{tab:disk1e6}
	\small\setlength{\tabcolsep}{4pt}
	\begin{tabular}{cccccc}
		\toprule
		\multicolumn{3}{c}{Eigenvalue $\lambda$} & \multicolumn{3}{c}{Frequency $z=\sqrt{\lambda}$} \\
		\midrule
		Numerical & Reference & Error & Numerical & Reference & Error \\
		\midrule
    1.907547503 & 1.8951318638  & 1.24$\times 10^{-2}$ & 1.3811399288  & 1.3766378840  & 4.50$\times 10^{-3}$ \\
    6.647900576 & 5.7831859629  & 8.65$\times 10^{-1}$ & 2.5783522986  & 2.4048255577  & 1.74$\times 10^{-1}$ \\
    14.68201634 & 14.6819706421  & 4.57$\times 10^{-5}$ & 3.8317119338  & 3.8317059702  & 5.96$\times 10^{-6}$ \\
    14.68201690 & 14.6819706421  & 4.63$\times 10^{-5}$ & 3.8317120059  & 3.8317059702  & 6.04$\times 10^{-6}$ \\
    26.37476218 & 26.3746164272  & 1.46$\times 10^{-4}$ & 5.1356364920  & 5.1356223018  & 1.42$\times 10^{-5}$ \\
    26.37476269 & 26.3746164272  & 1.46$\times 10^{-4}$ & 5.1356365418  & 5.1356223018  & 1.42$\times 10^{-5}$ \\
    32.00563798 & 30.4712623437  & 1.53$\times 10^{0}$ & 5.6573525590  & 5.5200781103  & 1.37$\times 10^{-1}$ \\
	40.70682765 & 40.7064658182  & 3.62$\times 10^{-4}$ & 6.3801902519  & 6.3801618959  & 2.84$\times 10^{-5}$ \\
	40.70682892 & 40.7064658182  & 3.63$\times 10^{-4}$ & 6.3801903518  & 6.3801618959  & 2.85$\times 10^{-5}$ \\
	49.21887199 & 49.2184563217  & 4.16$\times 10^{-4}$ & 7.0156162947  & 7.0155866698  & 2.96$\times 10^{-5}$ \\
		\bottomrule
	\end{tabular}
\end{table}

Table~\ref{tab:disk1e6} lists the first ten computed eigenvalues. The first eigenvalue corresponds to the Minnaert branch and is close to 
the characterization \eqref{eq:asy}  in Theorem \ref{thm:main}. The error is
\[
|z_{1,\mathrm{num}}-\sqrt{2}|\approx 3.31\times 10^{-2}. 
\]
The magnitude of error agrees with the estimation \eqref{eq:asy} with $\eta_2(\vep)\approx 0.07238$. Compared to the case $\vep=10^{-4}$,
the error between $z_{1,\mathrm{num}}$ and $\sqrt{2}$ is smaller, which is also agree with the theoretical prediction.

By solving the refined equation \eqref{eq:modify} numerically, we have the first reference value for the Minnaert frequency $z_{1,\mathrm{ref}}\approx 1.3766378840$. The error is
\[
|z_{1,\mathrm{num}}-z_{1,\mathrm{ref}}|\approx 4.502\times 10^{-3}.
\]

The remaining reference eigenvalues are background Dirichlet modes.  The double eigenvalues near $14.6820$, $26.3746$, and $40.7065$ are reproduced as close pairs, in agreement with the multiplicity preservation stated in Theorem~\ref{thm:main}.

\subsection{Square with a centered inclusion}\label{sec:num-square}

In this example, we take $D=(-1,1)^2$ and take the centered circular inclusion $B(0,\vep)$. The background eigenvalues of the homogeneous square are
\begin{align*}
\lambda_{m,n}=\frac{\pi^2}{4}(m^2+n^2),\qquad m,n\ge1.
\end{align*}
The Minnaert prediction for the reference inclusion remains $\lambda_M=2$.
\begin{figure}[!h]
	\centering
	\includegraphics[width=0.5\textwidth]{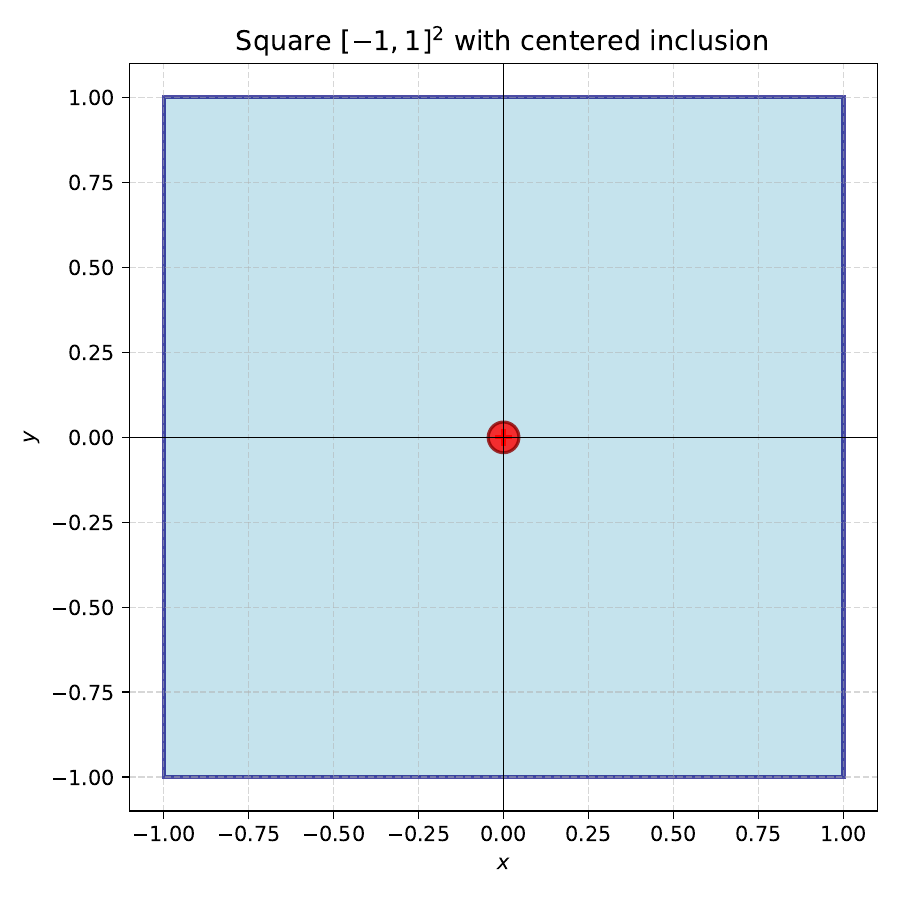}
	\caption{Square with centered inclusion.}
\end{figure}

\begin{table}[!htbp]
\centering
\caption{Square $[-1,1]^2$ with a centered inclusion: first ten computed eigenvalues and reference values.}
\label{tab:square}
\small\setlength{\tabcolsep}{4pt}
\begin{tabular}{cccccc}
\toprule
 \multicolumn{3}{c}{Eigenvalue $\lambda$} & \multicolumn{3}{c}{Frequency $z=\sqrt{\lambda}$} \\
\midrule
 Numerical & Reference & Error & Numerical & Reference & Error \\
\midrule
1.8225152367  & 1.7918704888   & 3.06$\times 10^{-2}$ & 1.3500056432  & 1.3386076680  & 1.14$\times 10^{-2}$ \\
6.1289146851  & 4.9348022005  & 1.19$\times 10^{0}$ & 2.4756644936  & 2.2214414691  & 2.54$\times 10^{-1}$ \\
12.3370654614  & 12.3370055014  & 6.00$\times 10^{-5}$ & 3.5124159010  & 3.5124073655  & 8.54$\times 10^{-6}$ \\
12.3370661744  & 12.3370055014  & 6.07$\times 10^{-5}$ & 3.5124160025  & 3.5124073655  & 8.64$\times 10^{-6}$ \\
19.7393689928  & 19.7392088022  & 1.60$\times 10^{-4}$ & 4.4429009659  & 4.4428829382  & 1.80$\times 10^{-5}$ \\
24.6742350454  & 24.6740110027  & 2.24$\times 10^{-4}$ & 4.9673166846  & 4.9672941329  & 2.26$\times 10^{-5}$ \\
26.4496848463  & 24.6740110027  & 1.78$\times 10^{0}$ & 5.1429257088  & 4.9672941329  & 1.76$\times 10^{-1}$ \\
32.0766334324  & 32.0762143035  & 4.19$\times 10^{-4}$ & 5.6636237015  & 5.6635866996  & 3.70$\times 10^{-5}$ \\
32.0766378952  & 32.0762143035  & 4.24$\times 10^{-4}$ & 5.6636240955  & 5.6635866996  & 3.74$\times 10^{-5}$ \\
41.9464325765  & 41.9458187046  & 6.14$\times 10^{-4}$ & 6.4766065634  & 6.4765591717  & 4.74$\times 10^{-5}$ \\
\bottomrule
\end{tabular}
\end{table}

The first eigenvalue  corresponds to the Minnaert branch and is close to 
the characterization \eqref{eq:asy}  in Theorem \ref{thm:main}. The error is 
\begin{align*}
|z_{1,\mathrm{num}}-\sqrt{2}|\approx 6.42\times 10^{-2}. 
\end{align*}
The magnitude of error agrees with the estimation \eqref{eq:asy} with $\eta_2(\vep)\approx 0.10857$.

By solving the refined equation \eqref{eq:modify} numerically, we have the first reference value for the Minnaert frequency $z_{1,\mathrm{ref}}\approx 1.338607668$. The error is
\begin{align*}
|z_{1,\mathrm{num}}-z_{1,\mathrm{ref}}|\approx 1.14\times 10^{-2}.
\end{align*}

The remaining reference eigenvalues are background Dirichlet modes.  As in the disk case, the double eigenvalues are recovered as the close pairs, confirming multiplicity preservation.

\subsection{Unit disk with an eccentric inclusion}\label{sec:num-ecc}

Then we consider the same unit disk but move the inclusion to $y_0=(0,0.9)$, i.e., $\Omega_\vep(y_0)=B((0,0.9),\vep)$. 
The background eigenvalues are same as the unit disk example in \ref{sec:num-disk}.
The same reference inclusion shape gives the same leading-order Minnaert value $\lambda_M=2$.
\begin{figure}[!h]
	\centering
	\includegraphics[width=0.5\textwidth]{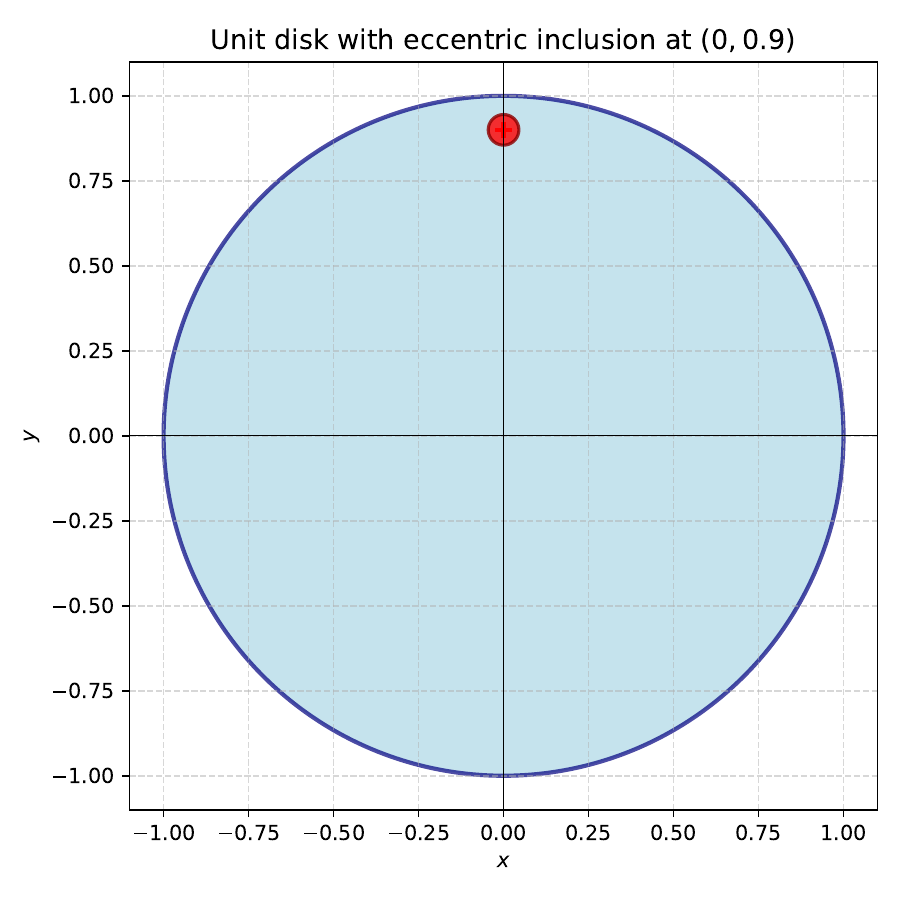}
	\caption{Unit disk with an eccentric inclusion.}
\end{figure}

\begin{table}[!htbp]
	\centering
	\caption{Unit disk with an eccentric inclusion at $(0,0.9)$: first ten computed eigenvalues and reference values.}
	\label{tab:ecc}
	\small\setlength{\tabcolsep}{4pt}
	\begin{tabular}{cccccc}
		\toprule
		\multicolumn{3}{c}{Eigenvalue $\lambda$} & \multicolumn{3}{c}{Frequency $z=\sqrt{\lambda}$} \\
		\midrule
		Numerical & Reference & Error & Numerical & Reference & Error \\
		\midrule
  2.4248836850  & 2.4597609603  & 3.49$\times 10^{-2}$ & 1.5572038033  & 1.5683625092  & 1.12$\times 10^{-2}$ \\
5.8114950183  & 5.7831859629  & 2.83$\times 10^{-2}$ & 2.4107042577  & 2.4048255577  & 5.88$\times 10^{-3}$ \\
14.6822056878  & 14.6819706421  & 2.35$\times 10^{-4}$ & 3.8317366412  & 3.8317059702  & 3.07$\times 10^{-5}$ \\
14.7786342544  & 14.6819706421  & 9.67$\times 10^{-2}$ & 3.8442989289  & 3.8317059702  & 1.26$\times 10^{-2}$ \\
26.3753632305  & 26.3746164272  & 7.47$\times 10^{-4}$ & 5.1356950095  & 5.1356223018  & 7.27$\times 10^{-5}$ \\
26.5263219670  & 26.3746164272  & 1.52$\times 10^{-1}$ & 5.1503710514  & 5.1356223018  & 1.47$\times 10^{-2}$ \\
30.5591832403  & 30.4712623437  & 8.79$\times 10^{-2}$ & 5.5280361106  & 5.5200781103  & 7.96$\times 10^{-3}$ \\
40.7082626645  & 40.7064658182  & 1.80$\times 10^{-3}$ & 6.3803027095  & 6.3801618959  & 1.41$\times 10^{-4}$ \\
40.9261537444  & 40.7064658182  & 2.20$\times 10^{-1}$ & 6.3973552148  & 6.3801618959  & 1.72$\times 10^{-2}$ \\
49.2210144628  & 49.2184563217  & 2.56$\times 10^{-3}$ & 7.0157689858  & 7.0155866698  & 1.82$\times 10^{-4}$ \\
		\bottomrule
	\end{tabular}
\end{table}

Table~\ref{tab:ecc} lists the first ten computed eigenvalues. The first eigenvalue corresponds to the Minnaert branch and  is close to 
characterization \eqref{eq:asy} in Theorem \ref{thm:main}. The error is 
\begin{align*}
|z_{1,\mathrm{num}}-\sqrt{2}|\approx 1.43\times 10^{-1}. 
\end{align*}
The magnitude of error agrees with the estimation \eqref{eq:asy} with $\eta_2(\vep)\approx 0.10857$.

By solving the refined equation \eqref{eq:modify} numerically, we have the first reference value for the Minnaert frequency $z_{1,\mathrm{ref}}\approx 1.5683625092$. The error is
\begin{align*}
|z_{1,\mathrm{num}}-z_{1,\mathrm{ref}}|\approx 1.12\times 10^{-2}.
\end{align*}

The remaining reference eigenvalues are background Dirichlet modes.  The double eigenvalues are reproduced as close pairs, in agreement with the multiplicity preservation stated in Theorem~\ref{thm:main}.

\subsection{Peanut-shaped domain with off-centered inclusion}\label{sec:num-peanut}

Finally we consider the peanut-shaped domain, which is defined by the radial function
\begin{align*}
r(t)=\sqrt{\cos^2t + 0.25\sin^2t}, \quad t\in[0,2\pi].
\end{align*}
The inclusion locates at $y_0=(-0.5,0)$, i.e., $\Omega_\vep(y_0)=B((-0.5,0),\vep)$. 
The background exact eigenvalues are unknown, we use the numerical solution without inclusion as reference.
The same reference inclusion shape gives the same leading-order Minnaert value $\lambda_M=2$.
\begin{figure}[!h]
	\centering
	\includegraphics[width=0.5\textwidth]{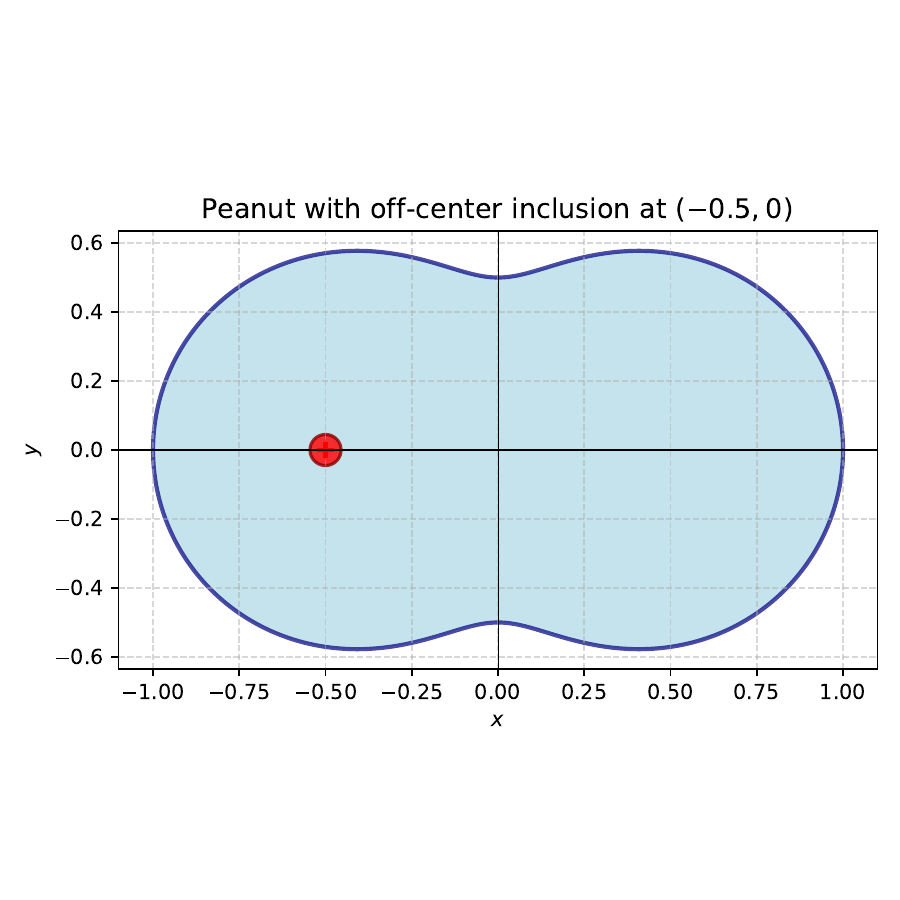}
	\caption{Peanut-shaped domain with off-centered inclusion.}
\end{figure}

\begin{table}[!htbp]
	\centering
	\caption{Peanut-shaped domain with off-centered inclusion at $(-0.5,0)$: first ten computed eigenvalues and reference values.}
	\label{tab:peanut}
	\small\setlength{\tabcolsep}{4pt}
	\begin{tabular}{cccccc}
		\toprule
		\multicolumn{3}{c}{Eigenvalue $\lambda$} & \multicolumn{3}{c}{Frequency $z=\sqrt{\lambda}$} \\
		\midrule
		Numerical & Reference & Error & Numerical & Reference & Error \\
		\midrule
		2.0685491541  &   2.0715159600  & 2.97$\times 10^{-3}$ & 1.4382451648  & 1.4392761931 & 1.03$\times 10^{-3}$ \\
		12.0219914650  & 11.2946647965  & 7.27$\times 10^{-1}$ & 3.4672743568  & 3.3607536054  & 1.07$\times 10^{-1}$ \\
		21.0459320120  & 19.2346023683  & 1.81$\times 10^{0}$ & 4.5875845509  & 4.3857271197  & 2.02$\times 10^{-1}$ \\
		34.8202024087  & 33.5786531401  & 1.24$\times 10^{0}$ & 5.9008645476  & 5.7947090643  & 1.06$\times 10^{-1}$ \\
		37.0086098319  & 37.0086078250  & 2.01$\times 10^{-6}$ & 6.0834702130  & 6.0834700480  & 1.65$\times 10^{-7}$ \\
		44.7410856447  & 44.7410816507  & 3.99$\times 10^{-6}$ & 6.6888777567  & 6.6888774582  & 2.99$\times 10^{-7}$ \\
		51.9723932006  & 51.8783267262  & 9.41$\times 10^{-2}$ & 7.2091881097  & 7.2026610864  & 6.53$\times 10^{-3}$ \\
		63.1558442938  & 63.1558411704  & 3.12$\times 10^{-6}$ & 7.9470651371  & 7.9470649406  & 1.97$\times 10^{-7}$ \\
		73.1138158999  & 73.1099891929  & 3.83$\times 10^{-3}$ & 8.5506617229  & 8.5504379533  & 2.24$\times 10^{-4}$ \\
		82.0211203153  & 81.0467575578  & 9.74$\times 10^{-1}$ & 9.0565512374  & 9.0025972673  & 5.40$\times 10^{-2}$ \\
		\bottomrule
	\end{tabular}
\end{table}

Table~\ref{tab:peanut} lists the first ten computed eigenvalues. The first eigenvalue corresponds to the Minnaert branch and  is close to 
the prediction in Theorem \ref{thm:main}. The error is 
\begin{align*}
|z_{1,\mathrm{num}}-\sqrt{2}|\approx 2.40\times 10^{-2}. 
\end{align*}
The magnitude of error agrees with the estimation \eqref{eq:asy} with $\eta_2(\vep)\approx 0.10857$.

By solving the refined equation \eqref{eq:modify} numerically, we have the first reference value for the Minnaert frequency $z_{1,\mathrm{ref}}\approx 1.4392761931$. The error is
\begin{align*}
|z_{1,\mathrm{num}}-z_{1,\mathrm{ref}}|\approx 1.03\times 10^{-3}.
\end{align*}
The remaining reference eigenvalues are background Dirichlet modes. 

The four numerical experiments confirm both parts of Theorem~\ref{thm:main}: an additional Minnaert eigenvalue branch appears near $\lambda_M=2$, while the background Dirichlet eigenvalues persist as perturbed clusters with preserved multiplicities. The errors of the Minnaert eigenvalue are consistent with the predicted $O(|\ln\vep|^{-1})$.

\section*{Acknowledgment}

The work of H. Diao is supported by National Natural Science Foundation of China  (Grant No. 12371422),  and the Fundamental Research Funds for the Central Universities, JLU. The work of M. Sini is supported by 2025 National Foreign Experts Program (Grant No. D20250157). The work of L. Li and M. Sini is supported by the Austrian Science Fund (FWF) grant P: 36942.
The work of Q. Zhai is supported by National Natural Science Foundation of China (12271208).

\end{document}